\documentclass[12pt, reqno]{amsart}
\usepackage{amsmath, amssymb, amsthm}
\usepackage{mathrsfs}
\usepackage{hyperref}
\usepackage{tikz}
\usepackage{mathtools}

\newtheorem{theorem}{Theorem}[section]
\newtheorem{lemma}[theorem]{Lemma}
\newtheorem{proposition}[theorem]{Proposition}
\newtheorem{corollary}[theorem]{Corollary}
\theoremstyle{definition}
\newtheorem{definition}[theorem]{Definition}
\newtheorem{remark}[theorem]{Remark}
\newtheorem{example}[theorem]{Example}
\newtheorem*{theorem*}{Theorem}
\newtheorem{observation}[theorem]{Observation}

\title{Coarsely Proper Actions of Topological Groups}
\author{Hussain Al-Rasheed}
\address{General Studies Department, Jubail Industrial College, Royal Commission for Jubail and Yanbu, Jubail Industrial City 31961, Kingdom of Saudi Arabia}
\email{rashedhs@rcjy.edu.sa}

\begin{document}

\begin{abstract}
We adapt the classical topological notions of properly discontinuous group actions to the framework of coarse geometry. By substituting the rigid constraints of local compactness and discreteness with uniform coarse equivalences, we resolve the topological obstructions identified by Kapovich. We establish large-scale characterisations for coarsely proper actions and coboundedness. As a primary application, we provide an affirmative answer to an open problem posed by Rosendal (Problem B.7) regarding the intrinsic coarse equivalence of cocompact subgroups in Polish groups. Furthermore, we construct large-scale fundamental sets via coarse Dirichlet domains, generalise the \v{S}varc-Milnor Lemma, and utilise the isometry group of the Urysohn universal space as an application.
\end{abstract}

\maketitle

\section{Introduction}
Classical geometric group theory relies heavily on properly discontinuous actions of discrete groups on locally compact spaces. For a discrete group $G$ acting on a locally compact space $X$, the action is properly discontinuous if for any compact subsets $K_1, K_2 \subseteq X$, the transporter set $\{g \in G \mid gK_1 \cap K_2 \neq \emptyset\}$ is finite. Kapovich \cite{Kapovich} provides a definitive treatment of these actions, clarifying the interplay between proper maps, dynamical relations, and the construction of fundamental domains.

However, this classical definition is insufficient for analysing the geometry of non-discrete topological groups. The shift from locally compact spaces to locally bounded Polish groups expands the scope of geometric group theory. Locally bounded groups encompass infinite-dimensional Banach spaces under addition, the unitary groups of infinite-dimensional Hilbert spaces, the isometry group of the Urysohn universal space, and the automorphism groups of countable structures. In these settings, local compactness fails, and non-trivial group actions inherently lack finite transporters. 

The foundations of coarse bornologies and coarsely proper actions for Polish and topological groups have been rigorously established in the texts of Roe \cite{Roe} and Rosendal \cite{Rosendal}. The purpose of this paper is to apply this framework to construct a translation dictionary that extends classical geometric group theory beyond the requirement of local compactness. Topological compactness is replaced by coarse boundedness, and finiteness in the group is replaced by boundedness with respect to the continuous left-invariant coarse structure. By generalising group actions to operate by uniform coarse equivalences, we resolve the topological obstructions identified by Kapovich. 

\subsection*{Main Results}

The first main objective of this paper is to establish the macroscopic equivalences for coarsely proper actions. We demonstrate that by requiring the group to act by uniform coarse equivalences, the action rigidifies the space sufficiently to preserve large-scale boundedness.

\begin{theorem*}[Theorem \ref{thm:equivalences}]
Let $G \curvearrowright X$ be a continuous action of a topological group $G$ on a coarse space $X$ by uniform coarse equivalences. Suppose either $G$ is a Polish group, or $X$ is a pseudometric space. Then the action is coarsely proper if and only if for every $x \in X$, the orbit map $\rho_x \colon G \to X$ is a coarsely proper map. Furthermore, when $G$ is locally bounded, this is equivalent to the absence of coarsely dynamically related bounded subsets in $X$.
\end{theorem*}

Parallel to our treatment of properness, we establish the following structural equivalences for coarse coboundedness, unifying the algebraic properties of the orbit map with the geometry of the quotient space.

\begin{theorem*}[Characterisations of Coarse Coboundedness; Theorem \ref{thm:cobounded_equivalences}]
Let $G \curvearrowright X$ be an action of a group $G$ on a coarse space $X$. The following statements are equivalent:
\begin{enumerate}
    \item The action is coarsely cobounded.
    \item For every $x \in X$, the orbit map $\rho_x \colon G \to X$ is coarsely surjective.
    \item The quotient space $X/G$, equipped with the quotient coarse structure, is coarsely bounded.
\end{enumerate}
\end{theorem*}

Building upon this structural rigidity, we synthesise these equivalences to obtain a large-scale generalisation of the fundamental theorem of geometric group theory, bypassing local compactness.

\begin{theorem*}[Coarse \v{S}varc-Milnor Lemma; Theorem \ref{thm:svarc_milnor}]
Let $G$ be a topological group acting continuously on a coarse space $X$ by uniform coarse equivalences, where either $G$ is Polish or $X$ is pseudometric. If the action is coarsely proper and coarsely cobounded, then for any $x \in X$, the orbit map $\rho_x \colon G \to X$ is a coarse equivalence. 
\end{theorem*}

To translate these algebraic equivalences into explicit geometric constructions—without relying on classical proper discontinuity or infimum-distance constraints—we construct large-scale fundamental sets via coarse Dirichlet domains. We prove that these domains satisfy the coarse analogue of the Koszul condition.

\begin{theorem*}[Theorem \ref{thm:dirichlet_koszul}]
Let $G \curvearrowright X$ be a coarsely proper action on a pseudometric space $(X, d_X)$ by isometries. For any basepoint $x_0 \in X$ and any constant $C > 0$, the coarse Dirichlet domain $D_{x_0}(C)$ is a coarse fundamental set that satisfies the coarse Koszul condition.
\end{theorem*}

Finally, as a primary application of this tiling framework, we provide an affirmative answer to an open problem posed by Rosendal (identified as Problem B.7 in Appendix B of \cite{Rosendal}) regarding the intrinsic coarse geometry of Polish groups.

\begin{theorem*}[Resolution of Problem B.7; Theorem \ref{thm:cocompact_subgroup}]
Let $G$ be a Polish group. If $H$ is a closed, cocompact subgroup of $G$, then $H$ equipped with its intrinsic coarse structure $\mathcal{E}_{\mathcal{P}(H)}$ is coarsely equivalent to $G$.
\end{theorem*}

\subsection*{Structure of the Paper}
In Section 2, we review the preliminaries of coarse geometry, isolating the macroscopic geometry of topological groups via continuous left-invariant pseudometrics. Section 3 refines the definition of topological coarse spaces and establishes the framework for actions by uniform coarse equivalences, analysing the properness and bornological properties of induced orbit maps. In Section 4, we prove the Characterisations of Coarse Proper Actions, connecting large-scale dynamical relations directly to the geometry of the transporter. Section 5 details the equivalences for Coarse Coboundedness and the resulting Coarse \v{S}varc-Milnor Lemma. Section 6 addresses the metrisability of coarse quotients and constructs large-scale fundamental sets via Dirichlet domains. Section 7 applies these structural equivalences to resolve Rosendal's Open Problem B.7. Finally, Section 8 utilises the isometry group of the Urysohn universal space as an application.

\section{Preliminaries}
The subset $\Delta \coloneqq \{(x, x) \mid x \in X\}$ of $X \times X$ is called the \emph{diagonal} of $X$. For $E, F \subseteq X \times X$, the \emph{product} of $E$ and $F$ is denoted by $E \circ F$ and defined by $E \circ F \coloneqq \{(x, z) \mid \exists y \in X, (x, y) \in E, (y, z) \in F\}$. The \emph{inverse} of $E$ is $E^{-1} \coloneqq \{(y, x) \mid (x, y) \in E\}$.

A \emph{coarse structure} on $X$ is a collection $\mathcal{E}$ of subsets $E \subseteq X \times X$ satisfying the following conditions:
\begin{enumerate}
    \item $\Delta \in \mathcal{E}$.
    \item If $E \subseteq F \in \mathcal{E}$, then $E \in \mathcal{E}$.
    \item If $E, F \in \mathcal{E}$, then $E \cup F, E \circ F, E^{-1} \in \mathcal{E}$.
\end{enumerate}

An element of $\mathcal{E}$ is called an \emph{entourage} or a \emph{controlled set}. The set $X$ equipped with a coarse structure $\mathcal{E}$ is called a \emph{coarse space} and denoted by $(X, \mathcal{E})$. We say that $(X, \mathcal{E})$ is \emph{coarsely connected} if $\{(x, y)\}$ is an entourage for all $x, y \in X$.

A subset $B \subseteq X$ is called \emph{coarsely bounded} or just \emph{bounded} if $B \times B \in \mathcal{E}$. In general, a collection $\mathcal{U}$ of subsets of $X$ is said to be uniformly bounded if $\displaystyle E_{\mathcal{U}} \coloneqq \bigcup_{U \in \mathcal{U}} U \times U \in \mathcal{E}$. Every element of $\mathcal{U}$ is a bounded subset of $X$. If, moreover, $\mathcal{U}$ is a cover of $X$, $\mathcal{U}$ is said to be a \emph{uniformly bounded cover} of $X$.

For $A \subseteq X$, and a subset $E \subseteq X \times X$, we define the $E$-ball around $A$ by $E[A] \coloneqq \{x \in X \mid \exists a \in A, (x, a) \in E\}$. Every subset $A \subseteq X$ naturally inherits a \emph{subspace coarse structure} $\mathcal{E}_A$ defined by restricting the entourages of $X$ to $A \times A$. Explicitly, $\mathcal{E}_A \coloneqq \{ E \cap (A \times A) \mid E \in \mathcal{E} \}$. Consequently, a subset $B \subseteq A$ is coarsely bounded in the subspace $(A, \mathcal{E}_A)$ if and only if it is coarsely bounded in the ambient space $(X, \mathcal{E})$.

A subset $\mathcal{E}' \subseteq \mathcal{E}$ is called a \emph{basis} (or fundamental system) for a coarse structure $\mathcal{E}$ if every entourage $E \in \mathcal{E}$ is contained in some $E' \in \mathcal{E}'$. 

By contrast, an arbitrary subset $\mathcal{S} \subseteq \mathcal{P}(X \times X)$ containing the diagonal \emph{generates} a coarse structure $\mathcal{E}$. In this generated structure, a set $E \in \mathcal{E}$ if and only if it is contained in a finite union of finite compositions of elements of $\mathcal{S}$ and their inverses. 

However, if a generating subset $\mathcal{S}$ is already closed under taking inverses, finite unions, and finite compositions (up to enlargement by other elements of $\mathcal{S}$), then $\mathcal{S}$ forms a basis for the coarse structure it generates. The coarse space $(X, \mathcal{E})$ is called \emph{monogenic} if its coarse structure is generated by a single entourage $E \in \mathcal{E}$. That is, for every entourage $F \in \mathcal{E}$, there exists an integer $n \ge 1$ such that $F \subseteq E^n$. In this case, $E$ is called a \emph{generating entourage}.

A \emph{bornology} $\mathcal{B}$ on a set $X$ is a family of subsets covering $X$ that is closed under taking subsets and finite unions. If the coarse space $(X, \mathcal{E})$ is coarsely connected, then the collection of coarsely bounded subsets of $X$ constitutes a bornology $\mathcal{B}$ on $X$. When $X=G$ is a group, we say that $\mathcal{B}$ is a \emph{group bornology} if it is additionally closed under group multiplication and inversion (i.e., $A B \in \mathcal{B}$ and $A^{-1} \in \mathcal{B}$ for all $A, B \in \mathcal{B}$). Elements of $\mathcal{B}$ are called bounded sets.

\subsection*{Main Examples}

\begin{example}
Let $\mathcal{B}$ be a bornology on $X$. Put $E_B \coloneqq \Delta_X \cup (B \times B)$. The collection $\{ E_B \mid B \in \mathcal{B}\}$ generates a connected coarse structure $\mathcal{E}_{\mathcal{B}}$ on $X$. Notice that:
\begin{itemize}
    \item If $x \in B$, then $E_B[x] = B$.
    \item If $x \notin B$, then $E_B[x] = \{x\}$.
\end{itemize}
When the bornology $\mathcal{B}$ on $X$ is the bornology of all finite subsets of $X$, the coarse structure is denoted by $\mathcal{E}_{\mathrm{fin}}$ called the \emph{finitary} coarse structure. 
\end{example}

\begin{example} 
A \emph{pseudometric} $d_X$ on $X$ is a symmetric non-negative real-valued function $d_X \colon X \times X \to \mathbb{R}_{\ge 0}$ satisfying the triangle inequality where $d_X(x, x)=0$ (but allowing $d_X(x, y)=0$ for $x \neq y$). The pair $(X, d_X)$ is called a \emph{pseudometric space}. 
For a pseudometric space $(X, d_X)$ and a non-negative real $r \ge 0$, put: 
\[ E_r \coloneqq \{(x, y) \in X \times X \mid d_X(x, y) \le r\}. \]
The collection $\mathcal{E}_{d_X} \coloneqq \{E_r \mid r \in \mathbb{R}_{\ge 0}\}$ is a coarse structure on $X$ induced by its pseudometric. The pair $(X,\mathcal{E}_{d_X})$ is a connected coarse space. The bounded set $E_r[x]$ is the closed $r$-ball centered at $x$. The family $\{E_n \mid n \in \mathbb{N}_{\ge 0}\}$ is a basis for the coarse structure $\mathcal{E}_{d_X}$. In general, a coarse structure $\mathcal{E}$ on a set $X$ is called \emph{pseudometrisable} if $\mathcal{E}$ is generated by $\mathcal{E}_{d_X}$ for some pseudometric $d_X$ on $X$. Indeed, a coarse structure $\mathcal{E}$ on a set $X$ is pseudometrisable if and only if it is countably generated \cite[Theorem 2.55]{Roe}.
\end{example}

\begin{example}
Let $\mathcal{B}$ be a group bornology on a group $G$. For $B \in \mathcal{B}$, we define:
\[ E_B \coloneqq \left\{(x, y) \in G \times G \mid y^{-1} x \in B\right\}. \]
The coarse structure generated by the collection $\{E_B \mid B \in \mathcal{B}\}$ is called the coarse structure induced by the bornology $\mathcal{B}$. Notice that under the definition $E_B[x] \coloneqq \{y \in G \mid (y,x) \in E_B\}$, this yields $E_B[x] = x B$. Furthermore, $E_B$ is left-invariant in the sense that $(x,y) \in E_B$ if and only if $(gx, gy) \in E_B$, for all $g \in G$.
\end{example}
 
While classical geometric group theory typically relies on the bornology of relatively compact subsets to define boundedness, this convention is often overly restrictive outside of locally compact spaces. The coarse structure generated by the bornology of relatively compact subsets is denoted by $\mathcal{E}_{\mathrm{com}}$. Consequently, we will follow the coarse geometric framework established by Rosendal \cite{Rosendal} for general topological groups.

\begin{example}
Let $(X,\mathcal{E}_X)$ and $(Y,\mathcal{E}_Y)$ be coarse spaces. For any $E_X \in \mathcal{E}_X$ and $E_Y \in \mathcal{E}_Y$, consider the set:
\[ W \coloneqq \{ ((x_1, y_1), (x_2, y_2)) \in (X \times Y) \times (X \times Y) \mid (x_1, x_2) \in E_X, \ (y_1, y_2) \in E_Y \} \]
The collection of all such sets $W$ generates a coarse structure $\mathcal{E}_{X \times Y}$ called the \emph{product coarse structure} on $X \times Y$.
\end{example}

\begin{example}
Let $(X, \mathcal{E}_X)$ be a coarse space equipped with an action of a group $G$, and let $\pi \colon X \to X/G$ denote the canonical projection map given by $\pi(x) \coloneqq [x]$. For any entourage $E \in \mathcal{E}_X$, consider the pushforward set:
\[ \tilde{E} \coloneqq (\pi \times \pi)(E) = \{([x], [y]) \mid \exists x' \in [x], y' \in [y] \text{ such that } (x', y') \in E \} \]
The collection of all such sets $\tilde{E}$ generates a coarse structure $\mathcal{E}_{X/G}$ called the \emph{quotient coarse structure} on the orbit space $X/G$. Under this induced structure, a subset $\tilde{B} \subseteq X/G$ is coarsely bounded if and only if there exists a coarsely bounded subset $B \subseteq X$ such that $\tilde{B} \subseteq \pi(B)$.
\end{example}

\subsection*{Left-coarse structure \texorpdfstring{$\mathcal{E}_\mathcal{P}$}{E\_P}}

Let $G$ be a topological group. A pseudometric $d$ on $G$ is a symmetric non-negative real-valued function $d \colon G \times G \to \mathbb{R}_{\ge 0}$ satisfying the triangle inequality where $d(x, x)=0$ (but allowing $d(x, y)=0$ for $x \neq y$). A pseudometric $d$ on $G$ is \emph{left-invariant} if $d(gx, gy)=d(x,y)$ for all $x,y,g \in G$. If moreover, the pseudometric $d \colon G \times G \to \mathbb{R}_{\ge 0}$ is a continuous function, we say that $d$ is a \emph{continuous}, left-invariant pseudometric on $G$. 

The family $\mathcal{P}$ of all continuous left-invariant pseudometrics on $G$ is non-empty as it contains the zero pseudometric. The family $\mathcal{P}$ is generated by the left-uniformity of the topological group $G$. By the classical results of Weil \cite{Weil}, the left-uniformity of any topological group $G$, and consequently its underlying topology, is entirely generated by the directed family of continuous, left-invariant pseudometrics. Specifically, Weil's theorem guarantees that for every open neighbourhood $U$ of the identity $1_G$, there exists a continuous left-invariant pseudometric $d_U \in \mathcal{P}$ such that the open unit ball is contained within $U$:
\[ \left\{x \in G \mid d_U(1_G, x) < 1\right\} \subseteq U. \]

While Weil utilised $\mathcal{P}$ to describe the microscopic structure of $G$, Rosendal shows it also determines the macroscopic geometry \cite{Rosendal}. By shifting focus from the local separation of points to their large-scale boundedness, the left-coarse structure $\mathcal{E}_\mathcal{P}$ is defined as the intersection of the coarse structures generated by each $d \in \mathcal{P}$. Thus, $\mathcal{P}$ serves as the structural bridge, translating the local topological richness of $G$ into its global coarse bornology. A subset $B \subseteq G$ is bounded if it has a finite diameter with respect to every $d \ \mathcal{P}$. The group bornology of all bounded subsets of $G$ with respect to the family $\mathcal{P}$ is denoted by $\mathcal{B}_\mathcal{P}$.

\begin{definition} \label{def:rosendal_pseudometric}
Following Rosendal \cite{Rosendal}, let $G$ be a topological group and let $\mathcal{P}$ denote the family of all continuous, left-invariant pseudometrics on $G$. We define the left-coarse structure $\mathcal{E}_\mathcal{P}$ on $G$ explicitly as the family of all entourages whose displacement is uniformly bounded across all continuous pseudometric scales:
\[ \mathcal{E}_\mathcal{P} \coloneqq \left\{ E \subseteq G \times G \;\middle|\; \forall d \in \mathcal{P}, \ \sup_{(x,y) \in E} d(x,y) < \infty \right\}. \]
As demonstrated by Rosendal \cite[Proposition 2.13]{Rosendal}, this coarse structure translates into the framework of group bornologies; this coarse structure is generated by the base of entourages $E_B \coloneqq \{(x,y) \in G \times G \mid x^{-1}y \in B\}$, where $B \subseteq G$ is any subset having a finite diameter with respect to every $d \in \mathcal{P}$. That is,
\[ \sup_{b \in B} d(1_G, b) < \infty \]
where $1_G$ is the identity element of $G$ \cite[Proposition 2.7]{Rosendal}. In this setting, $G$ is \emph{locally bounded} if it admits a coarsely bounded neighbourhood of the identity. $G$ is \emph{countably generated over every identity neighbourhood}, if for every identity neighbourhood $U$ there is a countable set $C_U \subseteq G$ so that $G=\langle U \cup C_U\rangle$. 
\end{definition}

\subsection*{Standing Convention on Topological Groups}
Unless explicitly stated otherwise, all topological groups $G$ considered throughout the remainder of this manuscript are assumed to be equipped with their left-coarse structure $\mathcal{E}_\mathcal{P}$ induced by the family $\mathcal{P}$ of all continuous, left-invariant pseudometrics on $G$.

For Polish groups, which are separable and thus automatically countably generated over every identity neighbourhood, this metric definition of coarse boundedness admits an algebraic characterisation. 

\begin{lemma}[{\cite[Proposition 2.7]{Rosendal}}] \label{lem:algebraic_boundedness}
Let $G$ be a Polish group. A subset $B \subseteq G$ is coarsely bounded if and only if for every open neighbourhood $V$ of the identity, there exists a finite set $F \subseteq G$ and an integer $k \ge 1$ such that $B \subseteq (FV)^k$.
\end{lemma}

\begin{proof}
Because $G$ is a Polish group, it contains a countable dense subset, ensuring that $G$ is countably generated over any neighbourhood of the identity. Under this condition, the equivalence between the pseudometric boundedness and this algebraic generation property is established in \cite[Proposition 2.7]{Rosendal}.
\end{proof}

\begin{lemma}[{\cite[Proposition 2.7]{Rosendal}}] \label{lem:pseudometric_boundedness}
Let $G$ be an arbitrary topological group. A subset $B \subseteq G$ is coarsely bounded if and only if for every continuous left-invariant pseudometric (écart) $d$ on $G$, the set $B$ has finite diameter, meaning:
\[ \sup_{b \in B} d(1_G, b) < \infty \]
where $1_G$ is the identity element of $G$.
\end{lemma}

\begin{lemma} \label{lem:Compar-Struc} 
Let $G$ be a topological group, and let $\mathcal{E}_{\mathrm{com}}$ denote the coarse structure generated by the bornology of relatively compact sets. Then $\mathcal{E}_{\mathrm{com}} \subseteq \mathcal{E}_{\mathcal{P}}$. 

Furthermore:
\begin{enumerate}
    \item If $G$ is a Polish group that is locally bounded but not locally compact, this inclusion is strict ($\mathcal{E}_{\mathrm{com}} \subsetneq \mathcal{E}_{\mathcal{P}}$).
    \item If $G$ is a locally compact Polish group, the coarse structures coincide ($\mathcal{E}_{\mathrm{com}} = \mathcal{E}_{\mathcal{P}}$).
\end{enumerate}
\end{lemma}

\begin{proof}
Let $E_K \in \mathcal{E}_{\mathrm{com}}$ be an entourage generated by a relatively compact subset $K \subseteq G$. Because continuous maps preserve compactness, the evaluation map $f_d \colon G \to \mathbb{R}$ defined by $f_d(g) = d(1_G, g)$ ensures that $K$ has bounded displacement with respect to every continuous left-invariant pseudometric $d \in \mathcal{P}$. Thus, $K \in \mathcal{B}_{\mathcal{P}}$, proving the general inclusion $\mathcal{E}_{\mathrm{com}} \subseteq \mathcal{E}_{\mathcal{P}}$ for any topological group.

To prove (1), assume $G$ is a Polish group that is locally bounded but not locally compact. Because $G$ is locally bounded, there exists an open neighbourhood of the identity, $V$, that is coarsely bounded with respect to $\mathcal{P}$. Thus, $V \in \mathcal{B}_{\mathcal{P}}$, and $V \notin \mathcal{B}_{\mathrm{com}}$.

To prove (2), assume $G$ is a locally compact Polish group. By definition, there exists an open relatively compact neighbourhood of the identity, $V$. Let $B \in \mathcal{B}_{\mathcal{P}}$ be an arbitrary coarsely bounded set in the left-coarse structure. By Lemma \ref{lem:algebraic_boundedness}, because $G$ is Polish, there exists a finite set $F \subseteq G$ and an integer $k \ge 1$ such that $B \subseteq (FV)^k$. Since $V \subseteq \overline{V}$, we have the inclusion $B \subseteq (F\overline{V})^k$. Because $F$ is finite and $\overline{V}$ is compact, their algebraic product $(F\overline{V})^k$ is the continuous image of a compact set, and is therefore compact. Hence $B \in \mathcal{B}_{\mathrm{com}}$. This yields the reverse inclusion $\mathcal{E}_{\mathcal{P}} \subseteq \mathcal{E}_{\mathrm{com}}$, proving the structures coincide.
\end{proof}

\subsection*{Coarse Separation and Density}
In analogy to point-set topology, coarse geometry utilises large-scale analogues of separation and density to analyse the global structure of spaces and maps.

\begin{definition}
Let $(X, \mathcal{E})$ be a coarse space, $E \in \mathcal{E}$, and $A, C \subseteq X$. 
\begin{itemize}
    \item The subsets $A, C \subseteq X$ are \emph{coarsely disjoint} if for every entourage $M \in \mathcal{E}$, the intersection $M[A] \cap C$ is a coarsely bounded subset of $X$.
    \item $A$ is called \emph{$E$-coarsely dense} if $E[A]=X$, and it is simply called \emph{coarsely dense} if it is $E$-coarsely dense for some $E \in \mathcal{E}$.
\end{itemize}
\end{definition}

\begin{remark} \label{rem:r_net_connection}
In the specific case where $(X, d_X)$ is a pseudometric space and the coarse structure is generated by metric entourages $E_R = \{(x, y) \in X \times X \mid d_X(x, y) \le R\}$ for $R > 0$, an $E_R$-coarsely dense set corresponds exactly to an $R$-net. When working within pseudometric contexts, we will refer to these sets natively as \emph{$R$-nets}.
\end{remark}

\subsection*{Topological Coarse Spaces}
To formalise the interaction between topological and coarse structures, we introduce topological coarse spaces and proper coarse spaces.

\begin{definition}
Let $X$ be a set equipped with a topology $\tau$ and a coarse structure $\mathcal{E}$.
\begin{itemize}
    \item The topology $\tau$ is \emph{compatible} with $\mathcal{E}$ if $\mathcal{E}$ contains an open neighbourhood of the diagonal $\Delta_X$ (equivalently, if there exists a uniformly bounded open cover of $X$).
    \item The triple $(X, \tau, \mathcal{E})$ is called a \emph{topological coarse space} if $\tau$ is compatible with $\mathcal{E}$.
    \item A topological coarse space $(X, \tau, \mathcal{E})$ is called a \emph{proper coarse space} if every coarsely bounded subset of $X$ is relatively compact in $(X, \tau)$.
\end{itemize}
\end{definition}

\begin{observation}
Any coarse space $(X, \mathcal{E})$ equipped with the discrete topology is a topological coarse space.
\end{observation}

\begin{remark}
In \cite{Roe}, John Roe defines a topology and a coarse structure to be compatible if $(X, \tau, \mathcal{E})$ constitutes a proper coarse space. We decouple these concepts, distinguishing between a topological coarse space and a proper coarse space. This modularity allows us to specify precisely which geometric theorems rely purely on topological compatibility and which fundamentally require local compactness.
\end{remark}

\subsection*{Maps Between Coarse Spaces}
\begin{definition}
     Let $\left(X, \mathcal{E}_X\right)$ and $\left(Y, \mathcal{E}_Y\right)$ be coarse spaces, $A$ be a set, and $\phi \colon X \to Y$ be a map.
     \begin{itemize}
         \item The map $\phi$ is called \emph{coarsely proper} if for each bounded subset $B \subseteq Y$, then $\phi^{-1}(B)$ is a bounded subset of $X$.
         \item The map $\phi$ is called \emph{controlled} or \emph{bornologous} if for each $E \in \mathcal{E}_X$, then $(\phi \times \phi)(E) \in \mathcal{E}_Y$. Equivalently, $\phi \colon X \to Y$ is controlled if for all $E \in \mathcal{E}_X$, there exists $F \in \mathcal{E}_Y$ such that $\phi(E[x]) \subseteq F[\phi(x)]$ for all $x \in X$.
         \item The map $\phi$ is called \emph{coarse} if it is coarsely proper and controlled.
         \item The maps $\phi_1 \colon A \to Y$ and $\phi_2 \colon A \to Y$ are \emph{close}, and we write $\phi_1 \sim \phi_2$, if $\{(\phi_1(a),\phi_2(a)) \mid a \in A\} \in \mathcal{E}_Y$.
         \item The coarse spaces $X$ and $Y$ are called \emph{coarsely equivalent} if there exist coarse maps $\phi \colon X \to Y$ and $\psi \colon Y \to X$ such that $\phi \circ \psi$ is close to $\operatorname{id}_Y$ and $\psi \circ \phi$ is close to $\operatorname{id}_X$. The coarse map $\phi$ is called a \emph{coarse inverse} of the map $\psi$.
         \item The coarse map $\phi \colon X \to Y$ is called a \emph{coarse equivalence} if it has a coarse inverse.
         \item The map $\phi$ is called \emph{coarsely surjective} if there exists $F \in \mathcal{E}_Y$ such that $F[\phi(X)]=Y$.
         \item Following Banakh and Protasov \cite{BanakhProtasov}, $\phi \colon X \to Y$ is \emph{coarsely closed} if for any coarsely disjoint subsets $A, C \subseteq X$ such that $A$ is $\phi$-saturated (i.e., $A = \phi^{-1}(\phi(A))$), their images $\phi(A)$ and $\phi(C)$ are coarsely disjoint in $Y$.
\end{itemize}
\end{definition}

\begin{lemma}\label{subsp-cor-equi}
    Let $(X, \mathcal{E}_X)$ be a coarse space, and let $A \subseteq X$ be a subset equipped with the subspace coarse structure. Then $A$ is coarsely dense in $X$ if and only if the inclusion map $\iota \colon A \hookrightarrow X$ is a coarse equivalence.
\end{lemma}

\begin{proof}
    Assume $A$ is coarsely dense in $X$. By definition, there exists an entourage $E \in \mathcal{E}_X$ such that $X = E[A]$. Without loss of generality, we may assume $E$ is symmetric and contains the diagonal $\Delta_X$. 
    
   Because $X = E[A]$, for every $x \in X$, the intersection $E[x] \cap A$ is non-empty. Using the Axiom of Choice, we construct a projection map $p \colon X \to A$ by choosing exactly one element $p(x) \in E[x] \cap A$ for each $x \in X$. If $x \in A$, we choose $p(x) = x$. 

We must show that $p$ is bornologous and that $\iota$ and $p$ are coarse inverses. By construction, we are guaranteed that $(x, p(x)) \in E$ for all $x \in X$. Therefore, the composition $\iota \circ p$ is $E$-close to the identity map $\operatorname{id}_X$. Furthermore, since $p(a) = a$ for all $a \in A$, the reverse composition $p \circ \iota = \operatorname{id}_A$, which is trivially close to the identity on $A$.

To verify that $p$ is bornologous, let $M \in \mathcal{E}_X$ be an arbitrary entourage. If two points are close, say $(x, y) \in M$, we can evaluate the distance between their projections. By the composition of entourages, we have:
\[ (p(x), p(y)) \in E^{-1} \circ M \circ E \in \mathcal{E}_X. \]
Since the image of $p$ is contained in $A$, we explicitly have:
\[ (p \times p)(M) \subseteq (E^{-1} \circ M \circ E) \cap (A \times A). \]
Because $A$ is equipped with the subspace coarse structure, this intersection forms a valid entourage in $\mathcal{E}_A$. Thus, $p$ is bornologous. Since $\iota$ and $p$ are mutually bornologous coarse inverses, $\iota$ is a coarse equivalence.

Conversely, assume the inclusion map $\iota \colon A \hookrightarrow X$ is a coarse equivalence. By definition, there exists a bornologous coarse inverse $f \colon X \to A$ such that $\iota \circ f$ is coarsely close to $\operatorname{id}_X$. 

Choose a symmetric entourage $E \in \mathcal{E}_X$ such that $(x, (\iota \circ f)(x)) \in E$ for all $x \in X$. Because $(\iota \circ f)(x) = f(x) \in A$ and $E$ is symmetric, this directly implies that for all $x \in X$, $x \in E[f(x)] \subseteq E[A]$. Consequently, $X = E[A]$, and hence $A$ is coarsely dense.
\end{proof}

\begin{proposition}\label{prop:bounded_index}
Let $G$ be a topological group. A subgroup $H$, equipped with the subspace coarse structure, is coarsely equivalent to $G$ if and only if $G = H B$ for some bounded set $B \in \mathcal{B}$.
\end{proposition}

\begin{proof}
Assume $G = H B$ for some bounded set $B \in \mathcal{B}$. In the left-invariant coarse structure of a topological group, entourages are generated by bounded sets via $E_B = \{(x,y) \in G \times G \mid x^{-1}y \in B\}$. Evaluating the entourage neighbourhood of $H$ yields:
\begin{align*}
E_B[H] &= \{g \in G \mid \exists h \in H \text{ such that } (h,g) \in E_B\} \\
       &= \{g \in G \mid \exists h \in H \text{ such that } h^{-1}g \in B\} = H B.
\end{align*}
 Since we assumed $G = HB$, we directly obtain $G = E_B[H]$ which implies $H$ is a coarsely dense subspace. By Lemma \ref{subsp-cor-equi}, $H$ is coarsely equivalent to $G$.

Conversely, assume that $H$ and $G$ are coarsely equivalent. By Lemma \ref{subsp-cor-equi}, $H$ is coarsely dense in $G$, meaning there exists an arbitrary entourage $E \in \mathcal{E}_G$ such that $G = E[H]$. By the definition of the left-invariant coarse structure, it is contained within an entourage $E_B$ generated by some bounded set $B \in \mathcal{B}$. Therefore, $G = E_B[H]$. As established above, $E_B[H] = HB$, yielding $G = H B$.
\end{proof}

We provide an example of a coarsely proper map that is not coarsely closed. 

\begin{example}\label{ex:proper_not_closed}
Let the subspace $X = (\mathbb{Z}_{\ge 0} \times \{0\}) \cup (\{0\} \times \mathbb{N})$ be realised in $\mathbb{R}^2$ equipped with the taxi metric $d((x_1, y_1), (x_2, y_2)) = \lvert x_1 - x_2 \rvert + \lvert y_1 - y_2 \rvert$. This space represents two divergent discrete rays originating from a central origin.

Let $A = \mathbb{Z}_{\ge 0} \times \{0\}$ be the horizontal ray and $B = \{0\} \times \mathbb{N}$ be the vertical ray (excluding the origin). Because the distance between any $(n, 0) \in A$ and $(0, m) \in B$ is exactly $n + m$, the rays diverge, making $A$ and $B$ coarsely disjoint.

Equipping the target space $\mathbb{Z}_{\ge 0}$ with the coarse structure induced by the standard Euclidean metric, we define a map $f \colon X \to \mathbb{Z}_{\ge 0}$ explicitly by:
\[ f(n, m) \coloneqq  \begin{cases}  2n & \text{if } m = 0 \text{ (Branch } A\text{)} \\ 2m - 1 & \text{if } n = 0 \text{ and } m \in \mathbb{N} \text{ (Branch } B\text{)} \end{cases} \]
This map is coarsely proper, as the preimage of any bounded interval is bounded in $X$. Furthermore, the set $A$ is $f$-saturated. However, the image $f(A)$ consists of all non-negative even integers, and $f(B)$ consists of all positive odd integers. Because these images alternate with a metric distance of one everywhere in the target space, they are entirely merged by the entourage $E_1$, meaning they are not coarsely disjoint. Thus, the map $f$ fails to be coarsely closed.
\end{example}

\section{Group Actions and Orbit Maps}

In lieu of strict isometries, we generalise our framework by requiring the group action to consist of uniform coarse equivalences.

\begin{definition}[\cite{MaRashedDydak}] \label{def:uniform_coarse_action}
Let a group $G$ act on a coarse space $(X, \mathcal{E})$. 
\begin{itemize}
    \item The action $\mu \colon G \times X \to X$, where $\mu(g, x) = gx$, is \emph{continuous} if the function $\mu$ is a continuous map with respect to the product topology on $G \times X$ and the topology on $X$. 
    \item The action is by \emph{uniform coarse equivalences} if for every entourage $E \in \mathcal{E}$, the family $\{gE\}_{g \in G}$ where: 
    \[ gE \coloneqq \{(gx, gy) \mid (x,y) \in E\} \]
    is uniformly controlled, meaning there exists a master entourage $M_E \in \mathcal{E}$ such that $\bigcup_{g \in G} gE \subseteq M_E$.
\end{itemize}
\end{definition}

\begin{example}
Consider the group $G = \mathbb{Z}$ acting on the space $X = \mathbb{R}$. Endow $\mathbb{R}$ with the perturbed metric:
\[ d(x, y) = \lvert x - y \rvert + \lvert \cos(x) - \cos(y) \rvert \]
The standard translation action of $G$ given by $nx = x + n$ is \emph{not} an isometry under $d$, because the oscillatory term changes depending on the integer shift $n$. However, we observe that for $E \subseteq E_r$ and for all $(x,y) \in E$ and $n \in \mathbb{Z}$:
\[ d(x+n, y+n) \le \lvert x - y \rvert + 2 \le d(x,y) + 2 \le r+2 \]
and therefore $\bigcup_{n \in \mathbb{Z}} nE \subseteq E_{r+2}$, which implies the group acts by uniform coarse equivalences.
\end{example}

In arbitrary coarse spaces, defining a quotient structure can lead to unwanted behaviour because the composition axiom of entourages might fail. However, requiring $G$ to act by uniform coarse equivalences ensures the quotient coarse structure on the orbit space $X/G$ is well-behaved.

\begin{proposition} \label{prop:quotient_structure}
Let $G \curvearrowright X$ be a continuous action of a group $G$ on a coarse space $X$ by uniform coarse equivalences. Let $\pi \colon X \to X/G$ be the canonical projection map. Then the pushforward collection:
\[ \mathcal{B} \coloneqq \{ (\pi \times \pi)(E) \mid E \in \mathcal{E}_X \} \]
forms a fundamental base for the coarse structure $\mathcal{E}_{X/G}$ on the orbit space $X/G$ generated by $\mathcal{B}$.
\end{proposition}

\begin{proof}
    We must show that for any two basic quotient entourages $\tilde{E}_1 \coloneqq (\pi \times \pi)(E_1)$ and $\tilde{E}_2 \coloneqq (\pi \times \pi)(E_2)$, their composition $\tilde{E}_1 \circ \tilde{E}_2$ is bounded by a single projected entourage from $X$.

    By choosing two symmetric master entourages $M_{E_1}$ and $M_{E_2}$ for $E_1$ and $E_2$ respectively, we obtain $\tilde{E}_1 \circ \tilde{E}_2 \subseteq (\pi \times \pi)(M_{E_1} \circ M_{E_2})$. Since $M_{E_1} \circ M_{E_2} \in \mathcal{E}_X$, its image under $\pi \times \pi$ belongs to $\mathcal{B}$, verifying the composition axiom and establishing that $\mathcal{B}$ forms a valid base for the coarse structure $\mathcal{E}_{X/G}$.
\end{proof}

If the action map $\mu \colon G \times X \to X$ defined by $\mu(g, x) = gx$ is controlled, then the action operates by uniform coarse equivalences. The converse is false. 

\begin{proposition} \label{prop:controlled_action}
Let $G \curvearrowright X$ be an action of a coarse group $G$ on a coarse space $X$. If the action map $\mu \colon G \times X \to X$ is controlled with respect to the product coarse structure on $G \times X$, then the action operates by uniform coarse equivalences.
\end{proposition}

\begin{proof}
Let $E \in \mathcal{E}_X$. We construct the product entourage $W \in \mathcal{E}_{G \times X}$ generated by the diagonal $\Delta_G \in \mathcal{E}_G$ and the entourage $E$:
\[ W \coloneqq \{ ((g, x), (g, y)) \in (G \times X) \times (G \times X) \mid g \in G, \ (x, y) \in E \} \]
Because the action map $\mu$ is controlled, the image of $W$ under $\mu \times \mu$ must be an entourage in $\mathcal{E}_X$. We compute this image:
\[ (\mu \times \mu)(W) = \{ (\mu(g,x), \mu(g,y)) \mid g \in G, \ (x, y) \in E \} = \bigcup_{g \in G} gE. \]
Setting the master entourage to be $M_E \coloneqq (\mu \times \mu)(W)$, we have $M_E \in \mathcal{E}_X$ such that $\bigcup_{g \in G} gE \subseteq M_E$. Therefore, the action operates by uniform coarse equivalences.
\end{proof}

An action by uniform coarse equivalences does not guarantee that the action map $\mu$ is controlled. Requiring $\mu$ to be controlled is a restrictive condition that demands bounded group elements displace the entire space uniformly. 

\begin{example}
Consider the topological group $G = (\mathbb{R}, +)$ acting on the pseudometric space $X = \mathbb{R}^2$ equipped with the standard Euclidean metric $d_X(u, v) = \lVert u - v \rVert_2$. The action is given by rotation around the origin: for an angle $\theta \in \mathbb{R}$ and a vector $v \in \mathbb{R}^2$, the action is $R_\theta v$, where $R_\theta$ is the standard $2 \times 2$ rotation matrix. 

Because every rotation is an isometry of the Euclidean metric, we have $d_X(R_\theta u, R_\theta v) = d_X(u, v)$. The action is by uniform coarse equivalences.

However, the action map $\mu \colon \mathbb{R} \times \mathbb{R}^2 \to \mathbb{R}^2$ is not a controlled map. Fix $0 < \epsilon < 1$, and consider the entourage $E_{\epsilon} \coloneqq \{(\alpha, \beta) \in \mathbb{R} \times \mathbb{R} \mid \lvert \alpha - \beta \rvert \le \epsilon\}$ in $\mathcal{E}_G$. Let $\Delta_X$ be the diagonal entourage in $X$. 

We construct the corresponding entourage $W \in \mathcal{E}_{G \times X}$ in the product space:
\[ W \coloneqq \{ ((\alpha, u), (\beta, v)) \in (\mathbb{R} \times \mathbb{R}^2) \times (\mathbb{R} \times \mathbb{R}^2) \mid (\alpha, \beta) \in E_\epsilon, \ (u, v) \in \Delta_X \} \]
Applying $\mu \times \mu$ to this entourage pairs points with their slightly rotated selves:
\[ (\mu \times \mu)(W) = \{ (R_\alpha v, R_\beta v) \mid \lvert \alpha - \beta \rvert \le \epsilon, \ v \in \mathbb{R}^2 \} \]

Fix an angle $0 < \delta \le \epsilon$. For any bounding entourage $F_r \in \mathcal{E}_X$, we can choose a vector $v \in \mathbb{R}^2$ with a sufficiently large magnitude such that $\lVert v \rVert_2 > \frac{r}{2 \sin(\delta/2)}$. Notice that for $\alpha = \delta$ and $\beta = 0$, we have:
\[ d_X(v, R_\delta v) = 2 \lVert v \rVert_2 \sin(\delta/2) > r \]
The image $(\mu \times \mu)(W)$ cannot be contained in any uniformly bounded entourage $F_r$. Thus, $\mu$ fails to be a controlled map.
\end{example}

The joint continuity of the action map $\mu \colon G \times X \to X$ guarantees the separate continuity of both the translation map $\gamma_g \colon X \to X$, defined by $x \mapsto gx$, and the orbit map $\rho_x \colon G \to X$. In the large-scale framework, if the action operates by uniform coarse equivalences, the translation map $\gamma_g$ is a coarse equivalence. However, without additional hypotheses, the orbit map $\rho_x$ is not even guaranteed to map coarsely bounded subsets of $G$ to coarsely bounded subsets of $X$—as the following example demonstrates.

\begin{example} \label{ex:non_bornologous_orbit}
Let the group $G=(\mathbb{Z},+)$ act on the space $X=\mathbb{R}$ (equipped with the coarse structure induced by the standard Euclidean metric) by: $n x = n+x$. The action operates by uniform coarse equivalences.

Equip $G=\mathbb{Z}$ with the bornology $\mathcal{B}$ consisting of all countable subsets. Under this valid coarse structure, the entire group $B = \mathbb{Z}$ is considered a coarsely bounded set.

Fix the basepoint $x=0 \in X$. The orbit of the coarsely bounded set $B$ under the action is $B 0 = \mathbb{Z} 0 = \mathbb{Z}$. But $\mathbb{Z}$ is an unbounded subset of $\mathbb{R}$. Thus, an action by uniform coarse equivalences does not guarantee a bornologous orbit map without additional topological hypotheses.
\end{example}

\subsection*{Standing Convention on Coarse Spaces and Group Actions}
Unless explicitly stated otherwise, we assume throughout the remainder of this manuscript that all coarse spaces $(X, \mathcal{E})$ are \emph{coarsely connected}, and that all group actions $G \curvearrowright X$ operate by \emph{uniform coarse equivalences}.

\begin{lemma}\label{Lem:Expan_orbit}
Let $G \curvearrowright X$ be an action, and let $x_0 \in X$. If the orbit map $\rho_{x_0} \colon G \to X$ is coarsely proper, then it is expanding.
\end{lemma} 

\begin{proof}
Take $E \in \mathcal{E}_X$. Since the action is by uniform coarse equivalences, there exists a symmetric entourage $M_E \in \mathcal{E}_X$ such that $\bigcup_{g \in G} gE \subseteq M_E$. 

Because $\rho_{x_0}$ is coarsely proper and $M_E[x_0]$ is bounded in $X$, the preimage $K \coloneqq \rho_{x_0}^{-1}(M_E[x_0])$ is bounded in $G$. Thus, $E_K \coloneqq \{(g,h) \in G \times G \mid g^{-1}h \in K\} \in \mathcal{E}_G$. Note that $g^{-1}h \in K$ if and only if  $\rho_{x_0}(g^{-1}h) = g^{-1}hx_0 \in M_E[x_0]$, which means $(g^{-1}hx_0, x_0) \in M_E$.

Evaluate the preimage of $E$ under $\rho_{x_0} \times \rho_{x_0}$:
\[
(\rho_{x_0} \times \rho_{x_0})^{-1}(E) = \{(g,h) \in G \times G \mid (gx_0, hx_0) \in E\}.
\]
If $(gx_0, hx_0) \in E$, applying the uniform action of $g^{-1}$ yields $(x_0, g^{-1}hx_0) \in g^{-1}E \subseteq M_E$. Because $M_E$ is symmetric, $(g^{-1}hx_0, x_0) \in M_E$. This establishes $(g,h) \in E_K$, proving $(\rho_{x_0} \times \rho_{x_0})^{-1}(E) \subseteq E_K \in \mathcal{E}_G$. 
\end{proof}

We now establish the unified conditions under which the orbit map preserves coarse boundedness and, more restrictively, when it constitutes a bornologous map.

\begin{lemma} \label{lem:bounded_orbits_unified}
Let $G \curvearrowright X$ be a continuous action of a topological group $G$ on a topological coarse space $X$ by uniform coarse equivalences. Suppose either:
\begin{enumerate}
    \item[(a)] $G$ is a Polish group, or
    \item[(b)] $(X, d_X)$ is a pseudometric space.
\end{enumerate}
If $K \subseteq G$ is a coarsely bounded subset of $G$, then for any $x \in X$, the orbit set $\rho_x(K) = \{kx \mid k \in K\}$ is a coarsely bounded subset of $X$.
\end{lemma}

\begin{proof}
Let $x_0 \in X$ be fixed. We proceed by establishing the boundedness of the orbit set $\rho_{x_0}(K) = \{kx_0 \mid k \in K\}$ under each respective hypothesis.

\noindent \textbf{Case (a): $G$ is a Polish group.}
 Let $U \in \mathcal{E}_X$ be an open entourage containing the diagonal. Because the action $G \curvearrowright X$ is continuous, the orbit map $\rho_{x_0} \colon G \to X$ defined by $\rho_{x_0}(g) = gx_0$ is continuous. Since the constant map $c_{x_0} \colon G \to X$ given by $c_{x_0}(g) = x_0$ is continuous, the product map $f_{x_0} \colon G \to X \times X$ defined by $f_{x_0}(g) = (gx_0, x_0)$ is continuous. The set $V \coloneqq \{g \in G \mid (gx_0, x_0) \in U\} = f_{x_0}^{-1}(U)$ is an open neighbourhood of the identity $1_G$. By Lemma \ref{lem:algebraic_boundedness}, there exists a finite subset $F \subseteq G$ and an integer $m \ge 1$ such that $K \subseteq (FV)^m$.

Consider the entourage $W \coloneqq \{(ax_0, x_0) \mid a \in F\}$. Since $U, W \in \mathcal{E}_X$, we can find master entourages $M_U, M_W \in \mathcal{E}_X$ such that $\bigcup_{g \in G} gU \subseteq M_U$ and $\bigcup_{g \in G} gW \subseteq M_W$. Notice that for all $v \in V$, $a \in F$, and $g \in G$, we have $(gvx_0, gx_0) \in M_U$ and $(gax_0, gx_0) \in M_W$.

Let $k \in K$. Since $K \subseteq (FV)^m$, there exist sequences of elements $a_1, \dots, a_m \in F$ and $v_1, \dots, v_m \in V$ such that $k = a_1 v_1 \dots a_m v_m$. We claim that $Kx_0 \subseteq (M_U \circ M_W)^m[x_0]$. To this end, define a finite sequence of group elements $\{g_i\}_{i=0}^m$ by $g_0 \coloneqq 1_G$ and $g_i \coloneqq g_{i-1} a_i v_i = a_1 v_1 \dots a_i v_i$ for $1 \le i \le m$. Notice that $g_m = k$. 

We observe that for all $1 \le i \le m$, the following relations hold:
\begin{align*}
    (g_i x_0, g_{i-1}a_i x_0) &= ((g_{i-1}a_i) v_i x_0, g_{i-1}a_i x_0) \in M_U,\text{ and} \\
    (g_{i-1}a_i x_0, g_{i-1} x_0) &\in M_W.
\end{align*}
These containments imply that $(g_i x_0, g_{i-1} x_0) \in M_U \circ M_W$.

By finite induction over $1 \le i \le m$, we can see that $(kx_0, x_0) = (g_m x_0, g_0 x_0) \in (M_U \circ M_W)^m$. Therefore, $Kx_0 \subseteq (M_U \circ M_W)^m[x_0]$, proving it is coarsely bounded.

\noindent \textbf{Case (b): $(X, d_X)$ is a pseudometric space.}
We construct a displacement function $D_{x_0} \colon G \times G \to [0, \infty)$ defined by the supremum of orbital shifts across the space: $D_{x_0}(g, h) \coloneqq \sup_{a \in G} d_X(agx_0, ahx_0)$.

To see that $D_{x_0}(g,h) \in [0,\infty)$, notice that the base pair $(gx_0, hx_0)$ belongs to some symmetric entourage $E \in \mathcal{E}_X$. Because the action operates by uniform coarse equivalences, the translated family $aE$ is uniformly bounded by a master entourage $M_E \in \mathcal{E}_X$. We can find $r>0$ such that $M_E \subseteq E_r$. That is, for all $a \in G$, $(agx_0, ahx_0) \in M_E \implies d_X(agx_0, ahx_0) \le r$. Therefore, $D_{x_0}(g,h) \le r < \infty$.

The function $D_{x_0}$ inherits symmetry and the triangle inequality directly from $d_X$. It is left-invariant because, for all $c \in G$: 
\[ D_{x_0}(cg, ch) = \sup_{a \in G} d_X(acgx_0, achx_0) = D_{x_0}(g, h) \] 
Because the action is continuous and the topology of $X$ is induced by $d_X$, $D_{x_0}$ constitutes a valid continuous left-invariant pseudometric on $G$.

By Lemma \ref{lem:pseudometric_boundedness}, the bounded set $K \subseteq G$ must have finite diameter under $D_{x_0}$. Thus, there is a radius $R < \infty$ such that $D_{x_0}(1_G, k) \le R$ for all $k \in K$. Evaluating the supremum at $a = 1_G$ yields $d_X(x_0, kx_0) \le R$ for all $k \in K$. The orbit $Kx_0$ is thus contained within a pseudometric ball of radius $R$, which is coarsely bounded.
\end{proof}

\begin{proposition} \label{prop:bornologous_orbits}
Under the hypotheses of Lemma \ref{lem:bounded_orbits_unified}, for any $x \in X$, the orbit map $\rho_x \colon G \to X$ is a bornologous map. In particular, if $\rho_x$ is a coarsely proper map, then it is coarse and expanding. 
\end{proposition}

\begin{proof}
Fix $x_0 \in X$. To show $\rho_{x_0}$ is bornologous, let $E \in \mathcal{E}_G$. By the definition of the coarse structure on $G$, there exists a coarsely bounded subset $K \subseteq G$ such that $E \subseteq E_K = \{(g,h) \in G \times G \mid g^{-1}h \in K\}$.

By Lemma \ref{lem:bounded_orbits_unified}, the orbit set $Kx_0 = \{kx_0 \mid k \in K\}$ is a coarsely bounded subset of $X$. Consequently, the subset $U \coloneqq \{(x_0, kx_0) \mid k \in K\} \subseteq \{x_0\} \times Kx_0$ is an entourage in $\mathcal{E}_X$. Because the action operates by uniform coarse equivalences, the family of translated entourages $gU = \{(gx_0, gkx_0) \mid k \in K\}$ is uniformly controlled by a master entourage $M_U \in \mathcal{E}_X$ such that $\bigcup_{g \in G} gU \subseteq M_U$.

Evaluating the image of $E$ under $\rho_{x_0} \times \rho_{x_0}$ yields:
\[ (\rho_{x_0} \times \rho_{x_0})(E) = \{(gx_0, hx_0) \mid (g,h) \in E\}. \]
We claim that $(\rho_{x_0} \times \rho_{x_0})(E) \subseteq M_U$. Indeed, for any $(g,h) \in E$, let $k = g^{-1}h \in K$. Then $(x_0, g^{-1}hx_0) \in U$, and hence $(gx_0, hx_0) \in gU \subseteq M_U$. Thus, the orbit map $\rho_{x_0}$ is bornologous.

Finally, if $\rho_{x_0}$ is additionally coarsely proper, its bornologous nature makes it a coarse map by definition, and it follows from Lemma \ref{Lem:Expan_orbit} that it is expanding.
\end{proof}

\begin{theorem} \label{thm:joint_boundedness}
Under the hypotheses of Lemma \ref{lem:bounded_orbits_unified}, if $K \subseteq G$ is a coarsely bounded subset of the group and $B \subseteq X$ is a coarsely bounded subset of the space, then their joint image $K B$ is a coarsely bounded subset of $X$.
\end{theorem}

\begin{proof}
Fix a basepoint $x_0 \in B$. By Lemma \ref{lem:bounded_orbits_unified}, the orbit set $K x_0$ is a coarsely bounded subset of $X$. 

Because $B$ is coarsely bounded, there exists an entourage $E \in \mathcal{E}_X$ such that $B \times B \subseteq E$. Since $x_0 \in B$, this implies $(b, x_0) \in E$ for all $b \in B$. The action operates by uniform coarse equivalences, guaranteeing the existence of a master entourage $M_E \in \mathcal{E}_X$ such that $(gb, g x_0) \in M_E$, and therefore $gb \in M_E[g x_0]$, for all $g \in G$ and $b \in B$. 

Taking the union over all elements $g \in K$ and $b \in B$ gives the inclusion:
\[ K B \subseteq M_E[K x_0]. \]
Because the base orbit $K x_0$ is bounded, the entourage neighbourhood $M_E[K x_0]$ is a coarsely bounded subset of $X$. Thus, the joint image $K B$ is coarsely bounded.
\end{proof}

\subsection*{Properness of Orbit Maps}

The primary objective of this framework is to characterise coarsely proper group actions on coarse spaces. As demonstrated in the subsequent section, the coarse properness of the global action is entirely governed by the coarse properness of its induced orbit maps. 

In general, a coarsely proper map is not automatically coarsely closed; it can easily fail by merging distant, disjoint regions of the space at infinity, as demonstrated in Example \ref{ex:proper_not_closed}. 

However, the algebraic rigidity of group actions ensures that orbit maps behave much better than arbitrary coarse maps. In the specific case of orbit maps, coarse properness intrinsically forces coarse closedness.

\begin{proposition} \label{prop:coarse_Bourbaki}
Let $G \curvearrowright X$ be a continuous action of a topological group $G$ on a topological coarse space $X$. Suppose either:
\begin{enumerate}
    \item[(a)] $G$ is a Polish group, or
    \item[(b)] $(X, d_X)$ is a pseudometric space.
\end{enumerate}
For any $x \in X$, if the orbit map $\rho_x \colon G \to X$ is coarsely proper, then it is coarsely closed.
\end{proposition}

\begin{proof}
We prove the stronger statement that the orbit map preserves coarse disjointness. Let $A, C \subseteq G$ be arbitrary coarsely disjoint subsets. We must show their images $\rho_x(A)$ and $\rho_x(C)$ are coarsely disjoint in $X$. 

Let $E \in \mathcal{E}_X$; we must show that $E[\rho_x(A)] \cap \rho_x(C)$ is a coarsely bounded subset of $X$. By Lemma \ref{Lem:Expan_orbit}, $\rho_x$ is expanding, and therefore $E' \coloneqq (\rho_x \times \rho_x)^{-1}(E) \in \mathcal{E}_G$. Since $A$ and $C$ are coarsely disjoint, the set $E'[A] \cap C$ is coarsely bounded in $G$. By Lemma \ref{lem:bounded_orbits_unified}, the image $\rho_x(E'[A] \cap C)$ is coarsely bounded in $X$. 

We claim that $E[\rho_x(A)] \cap \rho_x(C) \subseteq \rho_x(E'[A] \cap C)$. 

Let $z \in E[\rho_x(A)] \cap \rho_x(C)$. Since $z \in \rho_x(C)$, there exists $c \in C$ such that $\rho_x(c) = z$. Since $z \in E[\rho_x(A)]$, there exists some $a \in A$ such that $(\rho_x(c), \rho_x(a)) \in E$. The relation $(\rho_x(c), \rho_x(a)) \in E$ dictates that $(c, a) \in E'$. Under our entourage notation, this yields $c \in E'[a]$. It follows that $c \in E'[A]$.

Since $c \in C$ and $c \in E'[A]$, we have $c \in E'[A] \cap C$. Consequently, $z = \rho_x(c) \in \rho_x(E'[A] \cap C)$, which proves the claim. 

Because $E[\rho_x(A)] \cap \rho_x(C)$ is contained within the bounded set $\rho_x(E'[A] \cap C)$, it is coarsely bounded. Hence, $\rho_x(A)$ and $\rho_x(C)$ are coarsely disjoint, establishing that $\rho_x$ is coarsely closed.
\end{proof}

\section{Coarse Proper Actions}

We adapt Kapovich's \cite{Kapovich} topological equivalences for proper discontinuity to the coarse setting. Divergence to infinity is denoted $g_n \to_{\tau_{OB}} \infty$, meaning the sequence eventually leaves every coarsely bounded set in $G$. All groups in this section are locally bounded topological groups, meaning each group admits a coarsely bounded neighbourhood of the identity, serving as the coarse analogue of local compactness.

\begin{definition} Let $G \curvearrowright X$ be an action of a coarse group $G$ on a coarse space $X$: 
\begin{itemize}
    \item The transporter subset $(A \mid B)_G$ is defined as $(A \mid B)_G \coloneqq \{g \in G \mid g A \cap B \neq \emptyset\}.$
    \item Two coarsely bounded sets $A, B \subseteq X$ are \emph{coarsely dynamically related} if there exists a sequence $g_n \in G$ such that $g_n \to_{\tau_{OB}} \infty$ and $g_n A \cap B \neq \emptyset$ for all $n$; that is $\{g_n\}_{n\in\mathbb{N}} \subseteq (A \mid B)_G$.
    \item A coarsely bounded subset $B \subseteq X$ is \emph{coarsely wandering} if its transporter $(B \mid B)_G \coloneqq \{g \in G \mid gB \cap B \neq \emptyset\}$ is a coarsely bounded subset of $G$.
    \item The action is \emph{coarsely proper} if every coarsely bounded subset $B \subseteq X$ is coarsely wandering.
\end{itemize}
\end{definition}

\begin{theorem}[Characterisations of Coarse Proper Actions] \label{thm:equivalences}
Let $G \curvearrowright X$ be a continuous action of a topological group $G$ on a coarse space $X$. Suppose either:
\begin{enumerate}
    \item[(a)] $G$ is a Polish group, or
    \item[(b)] $(X, d_X)$ is a pseudometric space.
\end{enumerate}
The following conditions are equivalent:
\begin{enumerate}
    \item The action $G \curvearrowright X$ is coarsely proper.
    \item For every $x \in X$, the orbit map $\rho_x \colon G \to X$ is a coarsely proper map.
    \item For every coarsely bounded subset $B \subseteq X$ and every entourage $E \in \mathcal{E}_X$, the transporter of the entourage neighbourhood $(E[B] \mid E[B])_G$ is coarsely bounded in $G$.
\end{enumerate}
Furthermore, if $G$ is locally bounded, these conditions are equivalent to:
\begin{enumerate}
    \item[(4)] There are no coarsely bounded subsets $A, B \subseteq X$ that are coarsely dynamically related.
\end{enumerate}
\end{theorem}

\begin{proof}
$(1) \implies (2)$: Let $B \subseteq X$ be coarsely bounded. We examine the preimage $\rho_x^{-1}(B) = \{g \in G \mid gx \in B\}$. Let $B' = B \cup \{x\}$, which is a coarsely bounded subset of $X$. We observe that $g \in \rho_x^{-1}(B)$ implies $gx \in B \subseteq B'$ which implies $gB' \cap B' \neq \emptyset$. Thus, $\rho_x^{-1}(B) \subseteq (B' \mid B')_G$. By (1), the transporter $(B' \mid B')_G$ is coarsely bounded in $G$, hence the preimage $\rho_x^{-1}(B)$ is coarsely bounded. 

$(1) \implies (3)$: This implication is immediate.

$(3) \implies (1)$: Setting $E = \Delta_X$ (the diagonal entourage), we have $E[B] = B$. Thus, $(B \mid B)_G \subseteq (E[B] \mid E[B])_G$, making $(B \mid B)_G$ coarsely bounded.

Now, assuming further that $G$ is locally bounded, we establish the equivalence with condition (4):

$(2) \implies (4)$: Suppose there exist coarsely bounded subsets $A, B \subseteq X$ that are coarsely dynamically related. There exists a sequence $g_n \to_{\tau_{OB}} \infty$ such that $g_n A \cap B \neq \emptyset$. Thus, there exist elements $a_n \in A$ such that $g_n a_n \in B$ for all $n$. 

Fix a basepoint $x \in A$. Since $A \times A \in \mathcal{E}_X$ and $(x, a_n) \in A \times A$ for all $n$, there exists an entourage $M_A \in \mathcal{E}_X$ such that $(g_n x, g_n a_n) \in M_A$ for all $n$. 

Because $B$ is bounded, $E_B = B \times B \in \mathcal{E}_X$. Taking a fixed basepoint $y \in B$, the fact that $g_n a_n \in B \subseteq E_B[y]$ implies $(g_n a_n, y) \in E_B$. 

Hence $(g_n x, y) \in M_A \circ E_B$, for all $n$. This implies that the sequence $g_n x \in (M_A \circ E_B)[y]$. Thus, the orbit sequence $\{g_n x\}$ is coarsely bounded in $X$. Because $\rho_x$ is coarsely proper by (2), this implies the sequence $g_n$ is contained in a coarsely bounded subset of $G$, contradicting $g_n \to_{\tau_{OB}} \infty$.

$(4) \implies (1)$: We proceed by contrapositive. Suppose the action is not coarsely proper. Then there exists a coarsely bounded set $B \subseteq X$ such that the transporter $(B \mid B)_G$ is coarsely unbounded in $G$. By the definition of the coarse bornology on $G$, there exists a sequence $g_n \in (B \mid B)_G$ such that $g_n \to_{\tau_{OB}} \infty$. Because $g_n \in (B \mid B)_G$, it follows that $g_n B \cap B \neq \emptyset$ for all $n$. Setting $A = B$, the sets $B$ and $B$ are coarsely dynamically related, contradicting (4).
\end{proof}

\section{Coarse Coboundedness and the \v{S}varc-Milnor Lemma}

Having characterised coarsely proper actions, we now formulate the large-scale equivalent of a group translating a bounded set to cover an entire space.

\begin{definition}
An action $G \curvearrowright X$ is \emph{coarsely cobounded} if there exists a coarsely bounded set $B \subseteq X$ such that $GB = X$.
\end{definition}

We now establish the structural equivalences for coarse coboundedness.

\begin{theorem}[Characterisations of Coarse Coboundedness] \label{thm:cobounded_equivalences}
Let $G \curvearrowright X$ be an action of a group $G$ on a coarse space $X$. The following statements are equivalent:
\begin{enumerate}
    \item The action is coarsely cobounded.
    \item For every $x \in X$, the orbit map $\rho_x \colon G \to X$ is coarsely surjective.
    \item The quotient space $X/G$, equipped with the quotient coarse structure, is coarsely bounded.
\end{enumerate}
\end{theorem}

\begin{proof}
$(1) \implies (2)$: Let $x_0 \in X$ and suppose the action is coarsely cobounded. There exists a coarsely bounded subset $B \subseteq X$ such that $GB = X$. Without loss of generality, assume that $x_0 \in B$. Consider the entourage $E \coloneqq B \times B \in \mathcal{E}_X$. Observe that $B = E[x_0]$. Thus, $X = GB = \bigcup_{g \in G} g E[x_0]$.

By uniform coarse equivalences, there exists a master entourage $M_E \in \mathcal{E}_X$ such that if $(u, v) \in E$, then $(gu, gv) \in M_E$. In particular, for any $b \in B = E[x_0]$, we have $(b, x_0) \in E$ and hence $(gb, gx_0) \in M_E$, which means $gb \in M_E[gx_0]$. 

Consequently, $X = GB \subseteq \bigcup_{g \in G} M_E[gx_0] = M_E[Gx_0]$. This demonstrates that the orbit $Gx_0$ is $M_E$-coarsely dense in $X$. Since the image of the orbit map $\rho_{x_0}$ is coarsely dense, $\rho_{x_0}$ is coarsely surjective.

$(2) \implies (1)$: Suppose $\rho_{x_0}$ is coarsely surjective. Its image $Gx_0$ is coarsely dense in $X$, and thus there exists an entourage $E \in \mathcal{E}_X$ such that $X = E[Gx_0]$.

For an arbitrary point $x \in X = E[Gx_0]$, there exists $g \in G$ such that $x \in E[gx_0]$, which translates to $(x, gx_0) \in E$. Choose $M_E \in \mathcal{E}_X$ such that $\bigcup_{g \in G} gE \subseteq M_E$. The set $M_E[x_0]$ is coarsely bounded in $X$. Since $(x, gx_0) \in E \subseteq M_E$, applying the uniform action of $g^{-1}$ yields $(g^{-1}x, x_0) \in g^{-1}E \subseteq M_E$. This implies $g^{-1}x \in M_E[x_0]$, and therefore $x \in gM_E[x_0]$. This shows that $X = G M_E[x_0]$, proving the action is coarsely cobounded. 

$(1) \implies (3)$: Suppose the action is coarsely cobounded, meaning $GB = X$ for some coarsely bounded subset $B \subseteq X$. Consider the canonical projection map $\pi \colon X \to X/G$. By the definition of the quotient coarse structure, $\pi$ is a bornologous map. Since $B$ is coarsely bounded in $X$, its image $\pi(B)$ must be a coarsely bounded set in the quotient space $X/G$. 

Because the orbit of $B$ covers the space ($GB = X$), every orbit in $X$ intersects the bounded set $B$. This implies that the projection of the bounded set $B$ covers the entire quotient space, yielding $\pi(B) = X/G$. Therefore, the entire coarse space $X/G$ is coarsely bounded. 

$(3) \implies (1)$: Suppose $X/G$ is coarsely bounded. By the definition of the quotient coarse bornology, any coarsely bounded subset of $X/G$ must be contained in the image $\pi(B)$ of some coarsely bounded subset $B \subseteq X$. Since the entire space $X/G$ is bounded, we can find a bounded subset $B \subseteq X$ such that $X/G = \pi(B)$. This implies that for every $x \in X$, its orbit $[x] \in X/G$ belongs to $\pi(B)$. Therefore, there exists some $b \in B$ such that $[x] = [b]$, meaning $x = gb$ for some $g \in G$. Consequently, $X = GB$. Thus, the action is coarsely cobounded.
\end{proof}

\subsection*{\v{S}varc-Milnor Lemma}

The equivalences established in Theorem \ref{thm:equivalences} and Theorem \ref{thm:cobounded_equivalences} provide the necessary large-scale geometric conditions required for the \v{S}varc-Milnor Lemma. While classical versions of this lemma require proper discontinuity and local compactness, and recent generalisations address abstract large-scale groups \cite{MaRashedDydak}, we now formulate the result for topological groups. Using the bornological framework of Rosendal \cite{Rosendal}, we obtain the following coarse \v{S}varc-Milnor Lemma for continuous group actions.

\begin{theorem}[Coarse \v{S}varc-Milnor Lemma] \label{thm:svarc_milnor}
Let $G$ be a topological group acting continuously on a coarse space $X$. Suppose either:
\begin{enumerate}
    \item[(a)] $G$ is a Polish group, or
    \item[(b)] $(X, d_X)$ is a pseudometric space.
\end{enumerate}
If the action is coarsely proper and coarsely cobounded, then for any $x \in X$, the orbit map $\rho_x \colon G \to X$ is a coarse equivalence. 
\end{theorem}

\begin{proof}
Fix a basepoint $x \in X$. By Proposition \ref{prop:bornologous_orbits}, the orbit map $\rho_x$ is a bornologous map. By Theorem \ref{thm:equivalences}, because the action $G \curvearrowright X$ is coarsely proper, the orbit map $\rho_x$ is a coarsely proper map. By Theorem \ref{thm:cobounded_equivalences}, because the action is coarsely cobounded, the orbit map $\rho_x \colon G \to X$ is coarsely surjective.  
Therefore, $\rho_x$ is a coarse equivalence between $G$ and $X$.
\end{proof}

\section{Coarse Fundamental Domains and Dirichlet Tiling}

\subsection*{Metrisability of Coarse Quotients}

This section addresses the coarse and topological metrisability of the orbit space $X/G$ under continuous group actions. A fundamental requirement when extending geometric group theory to the coarse category is ensuring that the algebraic quotient structure remains compatible with the metric structure. To this end, we first establish that when a group acts by isometries, the canonical quotient coarse structure on the orbit space is compatible with the coarse structure induced by the quotient pseudometric. 

Following this structural compatibility, we analyse the metrisability of the quotient space by distinguishing two principal geometric regimes:
\begin{enumerate}
    \item \textbf{The Classical Metric Framework:} The underlying space $X$ is an arbitrary pseudometric space, which elevates to a Hausdorff metric space under the assumption that $X$ is a proper metric space and the action of $G$ is coarsely proper.
    \item \textbf{The Large-Scale Polish Framework:} The group $G$ is a locally bounded Polish group acting in a coarsely proper and coarsely cobounded manner on a general topological coarse space $X$. This allows us to establish coarse metrisability via a quotient pseudometric even when classical local compactness is entirely absent.
\end{enumerate}

As a primary application of this large-scale framework, we conclude the section by formalising the geometric structure of homogeneous spaces. We demonstrate that for closed subgroups of locally bounded Polish groups, the canonical quotient coarse structure on the coset space coincides with the Hausdorff-distance coarse structure.

\begin{lemma} \label{lem:quotient_coarse_structure}
Let $G$ act by isometries on a pseudometric space $(X, d_X)$. Then the orbit space $X/G$ naturally inherits a quotient pseudometric given by:
\[ d_{X/G}([x], [y]) = \inf_{g \in G} d_X(x, gy). \]
Moreover, the coarse structure on the orbit space $X/G$ induced by the quotient pseudometric $d_{X/G}$ coincides with the canonical quotient coarse structure induced by the projection map $\pi \colon X \to X/G$.
\end{lemma}

\begin{proof}
That $d_{X/G}$ defines a pseudometric follows from the assumption that $G$ acts by isometries, which guarantees symmetry and the triangle inequality on the orbit space. Let $\mathcal{E}_X$ denote the coarse structure on $X$ induced by $d_X$, which is generated by the entourages 
\[ E_R = \{ (x_1, x_2) \in X \times X \mid d_X(x_1, x_2) \le R \} \]
for $R > 0$. By definition, the canonical quotient coarse structure on $X/G$ is generated by the pushforwards of these entourages under $\pi \times \pi$:
\[ (\pi \times \pi)(E_R) = \{ ([x], [y]) \mid \exists x' \in [x], y' \in [y] \text{ such that } d_X(x', y') \le R \}. \]
Since $G$ acts by isometries, for any $x' = g_1 x$ and $y' = g_2 y$, we have $d_X(x', y') = d_X(g_1 x, g_2 y) = d_X(x, g_1^{-1} g_2 y)$. Thus, the condition that there exist such $x'$ and $y'$ is equivalent to the existence of some $g \in G$ such that $d_X(x, gy) \le R$.

Therefore, we can rewrite the projected entourage as:
\[ (\pi \times \pi)(E_R) = \{ ([x], [y]) \mid \inf_{g \in G} d_X(x, gy) \le R \}. \]
This is the entourage of radius $R$ defined by the quotient pseudometric $d_{X/G}$. Hence, the basic entourages of the quotient coarse structure and the basic entourages of the metric coarse structure induced by $d_{X/G}$ are identical, meaning the two coarse structures coincide.
\end{proof}

\begin{proposition} \label{prop:quotient_hausdorff}
Let a topological group $G$ act continuously by isometries on a proper metric space $(X, d_X)$. If the action of $G$ is proper, then $d_{X/G}$ is a metric, rendering the quotient space $X/G$ metrisable and hence Hausdorff.
\end{proposition}

\begin{proof}
To establish that $d_{X/G}$ is a metric under a proper action on a proper metric space, it suffices to show point separation. Suppose $[x], [y] \in X/G$ such that $d_{X/G}([x], [y]) = 0$. By definition of the infimum, there exists a sequence $(g_n)$ in $G$ such that:
\[ \lim_{n \to \infty} d_X(x, g_n y) = 0. \]
Therefore, the sequence $(g_n y)$ is eventually contained in the closed ball of radius $1$ centered at $x$, denoted $B_1(x)$. Because $X$ is a proper metric space, the closed ball $B_1(x)$ is compact. 

Because the action of $G$ is proper, the set of group elements $\{ g \in G \mid g\{y\} \cap B_1(x) \neq \emptyset \}$ is relatively compact in $G$. Thus, the sequence $(g_n)$ must contain a convergent subsequence $g_{n_k} \to g \in G$.

By the continuity of the action, $g_{n_k} y$ converges to $gy$. Since every subsequence of $(g_n y)$ must converge to the same limit $x$, we conclude that $x = gy$. This implies $[x] = [y]$, proving that $d_{X/G}$ separates points. 
\end{proof}

\begin{lemma}[{\cite[Proposition 2.14]{Rosendal}}] \label{lem:coarse_metrizability_polish}
Let $G$ be a Polish group. The coarse structure of $G$ is monogenic (i.e., coarsely metrisable) if and only if $G$ is locally bounded. Specifically, if $G$ is locally bounded, its left-coarse structure is generated by a single continuous left-invariant pseudometric.
\end{lemma}

\begin{proposition} \label{prop:coarse_metrizability}
Let $G$ be a locally bounded Polish group acting continuously on a topological coarse space $X$. If the action is coarsely proper and coarsely cobounded, then the space $X$ is coarsely metrisable, and the coarse orbit space $X/G$ inherits a quotient pseudometric.
\end{proposition}

\begin{proof}
By Lemma \ref{lem:coarse_metrizability_polish}, the locally bounded Polish group $G$ is a metrisable coarse space. Because the action of $G$ on $X$ is coarsely proper and coarsely cobounded, the Coarse \v{S}varc-Milnor Lemma (Theorem~\ref{thm:svarc_milnor}) implies that the orbit map $\rho_{x_0} \colon G \to X$ is a coarse equivalence. 

Since coarse metrisability is invariant under coarse equivalence, the space $X$ is coarsely metrisable. 
\end{proof}

Before we resolve the structural problem on homogeneous spaces, we formally identify the natural geometric structure on the coset space.

\begin{definition}
Let $G$ be a topological group equipped with its left-coarse structure, and let $H$ be a closed subgroup. The \emph{Hausdorff-distance coarse structure} on the coset space $H \backslash G$ is the coarse structure generated by the family of Hausdorff pseudometrics $d_{\mathrm{Haus}}$, obtained by pushing forward each continuous left-invariant pseudometric $d \in \mathcal{P}(G)$ via:
\[ d_{\mathrm{Haus}}(Hx, Hy) \coloneqq \sup_{a \in Hx} \inf_{b \in Hy} d(a, b). \]
\end{definition}

As a result of \ref{lem:quotient_coarse_structure}, we get a special case of Theorem 7 of Estrada and Rosendal \cite{EstradaRosendal}.

\begin{corollary}\label{cor:estrada_rosendal}
Let $G$ be a locally bounded Polish group and $H$ a closed subgroup of $G$. The canonical quotient coarse structure on the right coset space $H \backslash G$ coincides exactly with the Hausdorff-distance coarse structure. 

Consequently, if $H$ is a normal subgroup, the canonical quotient coarse structure on the quotient group $G/H$ is exactly its left-coarse structure.
\end{corollary}

\begin{proof}
Because $G$ is a locally bounded Polish group, its left-coarse structure is pseudometrisable (Lemma \ref{lem:coarse_metrizability_polish}) and generated by a single continuous left-invariant pseudometric $d_G$. 

Consider the action of $H$ on $G$ by left translations, given by $h \cdot x = hx$. Because $d_G$ is left-invariant, $d_G(h \cdot x, h \cdot y) = d_G(x, y)$, and hence $H$ acts on $(G, d_G)$ by isometries. The orbit space of this left action is the space of right cosets $H \backslash G$.

By Lemma \ref{lem:quotient_coarse_structure}, the canonical quotient coarse structure on $H \backslash G$ coincides with the coarse structure generated by the quotient pseudometric:
\[ d_{H \backslash G}(Hx, Hy) = \inf_{h \in H} d_G(x, hy). \]

We compare this to the Hausdorff pseudometric $d_{\mathrm{Haus}}$ between the subsets $Hx$ and $Hy$ in $G$. Because $d_G$ is left-invariant, the distance from any specific element $h_0 x \in Hx$ to the set $Hy$ is constant. Indeed:
\[ d_G(h_0 x, Hy) = \inf_{k \in H} d_G(h_0 x, k y) = \inf_{k \in H} d_G(x, h_0^{-1} k y) = \inf_{h \in H} d_G(x, hy). \]
Since this value is independent of the representative $h_1 \in H$, the supremum over all elements in $Hx$ collapses, yielding:
\[ d_{\mathrm{Haus}}(Hx, Hy) = \sup_{a \in Hx} \inf_{b \in Hy} d_G(a, b) = d_{H \backslash G}(Hx, Hy). \]
Thus, the quotient coarse structure and the Hausdorff-distance coarse structure are identical.

If $H$ is a normal subgroup, the right coset space $H \backslash G$ coincides with the quotient group $G/H$, and the Hausdorff-distance coarse structure corresponds exactly to the left-coarse structure on $G/H$.
\end{proof}

\begin{remark}
It is instructive to compare this result with Theorem 7 of Estrada and Rosendal \cite{EstradaRosendal}, which establishes the equivalence $\mathcal{E}_q(G/H) = \mathcal{E}_{\mathrm{Haus}}(G/H)$ for closed normal subgroups when $H$ is $\sigma$-cobounded in $G$. In Corollary \ref{cor:estrada_rosendal}, we strengthen the assumption on $G$ by assuming the entire ambient group $G$ is locally bounded. However, this global metric rigidity allows us to loosen the conditions on $H$, requiring only that $H$ is closed. Consequently, Corollary \ref{cor:estrada_rosendal} is applicable to any homogeneous space of a locally bounded Polish group.
\end{remark}

\subsection*{Coarse Fundamental Sets and Orbit Spaces}

Having established the categorical equivalences of coarsely proper and cobounded actions, we now address the constructive problem: \textit{Under what additional hypotheses on a coarse action $G \curvearrowright X$ can one construct a ``coarse fundamental domain'' whose coarse translates cover $X$ with uniformly bounded overlap?} 

In the classical setting, Kapovich shows that proper discontinuity provides the necessary rigidity for such a construction. In our coarse framework, the requirement of \emph{uniform coarse equivalence} provides the corresponding rigidity, allowing us to generalise the tiling properties of discrete groups to Polish and locally bounded groups.

Kapovich \cite[Proposition 34]{Kapovich} demonstrates that for proper actions, a fundamental set $F \subseteq X$ topologically mirrors the quotient space, yielding a homeomorphism $F/G \cong X/G$. We can establish a direct large-scale counterpart by defining fundamental sets bornologically and upgrading the homeomorphism to a coarse equivalence.

\begin{definition}
A subset $F \subseteq X$ is a \emph{coarse fundamental set} for the action $G \curvearrowright X$ if its orbit is $E$-coarsely dense for some entourage $E \in \mathcal{E}$, meaning that $E[G F] = X$. 
\end{definition}

\begin{proposition}\label{prop:coarse_invariance}
Let $G$ act on coarse spaces $X$ and $Y$. Suppose $h \colon X \to Y$ is a $G$-equivariant coarse equivalence. If $F \subseteq X$ is a coarse fundamental set for the action on $X$, then its image $h(F)$ is a coarse fundamental set for the action on $Y$.
\end{proposition}

\begin{proof}
By definition, since $F$ is a coarse fundamental set in $X$, there exists an entourage $E \in \mathcal{E}_X$ such that $E[G F] = X$. 

Because $h \colon X \to Y$ is a coarse equivalence, the image entourage $E_h = (h \times h)(E)$ belongs to $\mathcal{E}_Y$. Applying this to $h(G F)$ yields $E_h[h(G F)] = h(E[G F]) = h(X)$. 

Since $h$ is coarsely surjective, there exists an entourage $K \in \mathcal{E}_Y$ such that $K[h(X)] = Y$. Composing these entourages, we define $E' = K \circ E_h \in \mathcal{E}_Y$, which yields $E'[h(G F)] = Y$.  

Since $h$ is $G$-equivariant, the action commutes with the mapping, yielding $h(G F) = G h(F)$. It follows that $E'[G h(F)] = Y$. Thus, $h(F)$ is a coarse fundamental set in $Y$.
\end{proof}

\begin{proposition} \label{prop:coarse_prop34}
Let a group $G$ act on a pseudometric space $(X, d_X)$ by isometries. If $F \subseteq X$ is a coarse fundamental set, then the natural projection map $p \colon F \to X/G$ given by $p(x) = [x]_{X/G}$ induces a coarse equivalence $\tilde{p} \colon F/G \to X/G$ given by $\tilde{p}([x]_{F/G}) = [x]_{X/G}$.
\end{proposition}

\begin{proof}
We equip the restricted orbit space $F/G$ with the subspace pseudometric inherited from $X/G$, explicitly given by $d_{F/G}([x], [y]) = \inf_{g \in G} d_X(x, gy)$ for $x, y \in F$. Under this restricted metric, the induced map $\tilde{p} \colon F/G \to X/G$ is simply the inclusion map, which is an isometric embedding and therefore bornologous.

It remains to show that the image $\tilde{p}(F/G)$ is coarsely dense in $X/G$. Let $[z] \in X/G$ be an arbitrary orbit represented by $z \in X$. Because $F$ is a coarse fundamental set, $G F$ is an $R$-net in $X$ for some $R > 0$. Thus, there exists a point $y \in F$ and an element $g_0 \in G$ such that $d_X(z, g_0 y) \le R$. 

In the quotient space $X/G$, the distance between the orbit of $z$ and the orbit of $y$ satisfies:
\[ d_{X/G}([z], [y]) = \inf_{g \in G} d_X(z, gy) \le d_X(z, g_0 y) \le R. \]
Since $[y] \in \tilde{p}(F/G)$, this proves that the image of the restricted quotient space forms an $R$-net in $X/G$. Because $\tilde{p}$ is an isometric embedding with a coarsely dense image, it is a coarse equivalence. 
\end{proof}

\begin{definition}
Let $(X, d_X)$ be a pseudometric space, and $x_0 \in X$ be a fixed basepoint. For the constant $C > 0$, the \emph{coarse Dirichlet domain} centered at $x_0$ is defined as:
\[ D_{x_0}(C) = \{x \in X \mid d_X(x, x_0) \le d_X(x, gx_0) + C \text{ for all } g \in G\}. \]
\end{definition}

\begin{proposition} \label{prop:dirichlet}
Suppose the action $G \curvearrowright X$ is coarsely cobounded on a pseudometric space $(X, d_X)$. Then there exists a constant $R > 0$ such that the coarse Dirichlet domain $D_{x_0}(R)$ is a bounded and coarse fundamental set with $G D_{x_0}(R) = X$.
\end{proposition}

\begin{proof}
Because the action is coarsely cobounded, there exists $R > 0$ such that $G B_{R}(x_0) = X$. This implies the orbit $Gx_0$ is an $R$-net in $X$. 

Observe that $B_R(x_0) \subseteq D_{x_0}(R)$; since $G B_R(x_0) = X$, it follows that $G D_{x_0}(R) = X$. Second, we prove $D_{x_0}(R)$ is coarsely bounded. Let $y \in D_{x_0}(R)$. By definition, $d_X(y, x_0) \le d_X(y, gx_0) + R$ for all $g \in G$. Since $Gx_0$ is an $R$-net, there exists a specific $h \in G$ such that $d_X(y, hx_0) \le R$. Substituting this $h$ into the inequality gives $d_X(y, x_0) \le R + R = 2R$. Thus, $D_{x_0}(R) \subseteq B_{2R}(x_0)$.
\end{proof}

To control the large-scale intersections of fundamental sets, we introduce a coarse analogue of the Koszul condition.

\begin{definition}
Let $G \curvearrowright X$ be an action on a topological coarse space $(X, \tau, \mathcal{E})$. A coarse fundamental set $F \subseteq X$ satisfies the \emph{coarse Koszul condition} if there exists an entourage $E \in \mathcal{E}$ such that for every coarsely bounded subset $B \subseteq X$, the transporter $(E[F] \mid B)_G$ is a coarsely bounded subset of $G$.
\end{definition}

\begin{theorem} \label{thm:dirichlet_koszul}
Let $G \curvearrowright X$ be a coarsely proper action on a pseudometric space $(X, d_X)$ by isometries. For any basepoint $x_0 \in X$ and any constant $C > 0$, the coarse Dirichlet domain $D_{x_0}(C)$ is a coarse fundamental set that satisfies the coarse Koszul condition.
\end{theorem}

\begin{proof}
Let $C > 0$ be an arbitrary constant and $x_0 \in X$. First, we establish that $F = D_{x_0}(C)$ is a coarse fundamental set. Let $x \in X$ be an arbitrary element, and we aim to show that $x \in GF$. To this end, consider its orbit $Gx$, and the set of distances $\{d_X(gx, x_0) \mid g \in G\} \subseteq [0,\infty)$. Since $\{d_X(gx, x_0) \mid g \in G\}$ is bounded below, the infimum $d_{\mathrm{inf}} = \inf_{g \in G} d_X(gx, x_0)$ exists. Since $C > 0$, there exists an element $g_0 \in G$ such that $d_X(g_0 x, x_0) < d_{\mathrm{inf}} + C$. By the definition of the infimum, $d_{\mathrm{inf}} \le d_X(gx, x_0)$ for all $g \in G$. In particular:
\[ d_X(g_0 x, x_0) < d_{\mathrm{inf}} + C \le d_X(g^{-1} g_0 x, x_0) + C \]
Because $G$ acts by isometries, $d_X(g^{-1} g_0 x, x_0) = d_X(g_0 x, gx_0)$. Therefore, $g_0x \in D_{x_0}(C)$, which implies $x = g_0^{-1}(g_0 x) \in g_0^{-1} D_{x_0}(C) \subseteq GF$. Thus, $GF = X$.

Now, we prove $F$ satisfies the coarse Koszul condition. Fix an arbitrary thickening radius $r > 0$, and let $B \subseteq X$ be an arbitrary coarsely bounded set. We must show that the transporter $(B_r(F) \mid B)_G$ (where $B_r(F)$ is the metric $r$-neighbourhood of $F$) is coarsely bounded in $G$. To this end, it suffices to find bounded subsets $K_1, K_2 \subseteq X$ such that $(B_r(F) \mid B)_G \subseteq (K_1 \mid K_2)_G$.

Choose $R_B > 0$ such that $B \subseteq B_{R_B}(x_0)$. We claim that $(B_r(F) \mid B)_G \subseteq (B_{R}(x_0) \mid B_{R}(x_0))_G$, where $R = R_B + 2r + C$.

Indeed, $g \in (B_{R}(x_0) \mid B_{R}(x_0))_G$ if and only if there exists $x \in B_{R}(x_0)$ such that $gx \in B_{R}(x_0)$. Equivalently, $d_X(x, x_0) \le R$ and $d_X(gx, x_0) \le R$.

Let $g \in (B_r(F) \mid B)_G$. By definition, $gx \in B \subseteq B_{R_B}(x_0)$ for some $x \in B_r(F) = B_r(D_{x_0}(C))$. Consequently, $d_X(gx, x_0) \le R_B < R$, and there exists $y \in D_{x_0}(C)$ such that $d_X(x, y) \le r$. By the definition of the coarse Dirichlet domain, $d_X(y, x_0) \le d_X(y, hx_0) + C$ for all $h \in G$. Setting $h = g^{-1}$ yields $d_X(y, x_0) \le d_X(y, g^{-1}x_0) + C$.

Since $G$ acts by isometries, we can bound $d_X(y, g^{-1}x_0)$ as follows:
\begin{align*} 
    d_X(y, g^{-1}x_0) &= d_X(gy, x_0) \\ 
                      &\le d_X(gy, gx) + d_X(gx, x_0) \\ 
                      &= d_X(y, x) + d_X(gx, x_0) \\ 
                      &\le r + R_B. 
\end{align*}
Substituting this into the previous inequality gives: 
\[ d_X(y, x_0) \le r + R_B + C. \]
Finally, applying the triangle inequality for $x$:
\begin{align*} 
    d_X(x, x_0) &\le d_X(x, y) + d_X(y, x_0) \\ 
                &\le r + (r + R_B + C) \\ 
                &= R_B + 2r + C = R. 
\end{align*}
Because both $d_X(x, x_0) \le R$ and $d_X(gx, x_0) \le R_B < R$, we conclude $g \in (B_R(x_0) \mid B_R(x_0))_G$, completing the proof.
\end{proof}
\section{Resolution of Rosendal's Open Problem B.7}

The concept of a coarse fundamental set established in the previous section allows us to resolve structural obstructions in the coarse geometry of topological groups. When a closed subgroup $H$ is cocompact in a Polish group $G$, there exists a compact set $K \subseteq G$ that acts as a coarse fundamental Dirichlet domain for the coset space $G/H$ (as $K$ is compact in a Hausdorff space, it satisfies the requirement of being a closed subset). 

We now apply this tiling framework to resolve an open problem posed by Rosendal \cite{Rosendal} regarding the intrinsic coarse geometry of Polish groups.

\subsection*{Intrinsic vs. Subspace Boundedness}

In his foundational text \cite{Rosendal}, Rosendal identifies several open problems concerning the coarse geometry of Polish groups. Specifically, Problem B.7 (as numbered in Appendix B of preprint \cite{Rosendal}) asks whether a cocompact closed subgroup $H$ of a Polish group $G$ is necessarily coarsely equivalent to $G$ under the coarse structure $\mathcal{E}_{\mathcal{P}(H)}$ induced by its native pseudometrics. Our geometric framework provides a resolution to Problem B.7.

Before presenting the formal resolution, we highlight a structural property of topological groups. If $H$ is a subgroup of a coarsely bounded Polish group $G$, $H$ is bounded with respect to the subspace coarse structure. That is, for every continuous left-invariant pseudometric $d \in \mathcal{P}(G)$, the restriction $d|_H$ is bounded. However, the open problems require $H$ to satisfy these properties natively, meaning it must be evaluated with respect to the coarse structure $\mathcal{E}_{\mathcal{P}(H)}$ induced by the family of all continuous left-invariant pseudometrics $\mathcal{P}(H)$ on $H$.

While the restriction mapping $\mathcal{P}(G)|_H \subseteq \mathcal{P}(H)$ always holds, the reverse is generally false. A closed subgroup $H$ can possess unbounded pseudometrics in $\mathcal{P}(H)$ that fundamentally cannot be extended to a continuous left-invariant pseudometric on the overgroup $G$. 

\begin{example}
Consider the infinite symmetric group $G = \operatorname{Sym}(\mathbb{Z})$ equipped with the topology of pointwise convergence. As established by Bergman \cite[Theorem 6]{Bergman}, $G$ possesses Bergman's property: for every generating symmetric set $S$ of $G$, we have that $G=S^k$ for some natural number $k$. To see why this forces the group to be coarsely bounded, let $d \in \mathcal{P}(G)$ be any continuous left-invariant pseudometric. The sequence of symmetric sets $S_m = \{g \in G \mid d(e, g) < m\}$ exhausts the group. Because groups with Bergman's property possess uncountable cofinality—meaning they cannot be expressed as a countable union of proper subgroups—the ascending chain of subgroups $\langle S_m \rangle$ must stabilise \cite[Theorem 5]{Bergman}. Thus, for a sufficiently large $m$, the set $S_m$ generates $G$. By Bergman's property, there exists an integer $k$ such that $G = S_m^k$. The triangle inequality implies that $d(e, g) < k m$ for all $g \in G$. Consequently, every continuous left-invariant pseudometric on $G$ is globally bounded, making $G$ coarsely bounded \cite{Rosendal}. 

Let $H \cong \mathbb{Z}$ be the closed discrete subgroup generated by the shift operator $n \mapsto n+1$ on $\mathbb{Z}$. The subgroup $H$ admits the standard unbounded Euclidean metric $d_H(n, m) = \lvert n - m \rvert$, meaning $d_H \in \mathcal{P}(H)$. Since $G$ is coarsely bounded, $d_H$ cannot be extended to a continuous left-invariant pseudometric on $G$. Consequently, with respect to the subspace coarse structure inherited from $G$, the subgroup $H$ is completely bounded, but with respect to its intrinsic coarse structure $\mathcal{E}_{\mathcal{P}(H)}$, $H$ is unbounded.
\end{example}
\newpage
\subsection*{Formal Resolution}

\begin{definition} \label{def:cocompact}
Let $G$ be a topological group and $H$ be a closed subgroup of $G$. The subgroup $H$ is said to be \emph{cocompact} in $G$ if there exists a compact subset $K \subseteq G$ such that $G = HK$. 
\end{definition}

\begin{lemma} \label{lem:intrinsic_extrinsic}
Let $G$ be a Polish group equipped with the coarse structure $\mathcal{E}_{\mathcal{P}(G)}$ induced by $\mathcal{P}(G)$ and let $H$ be a closed cocompact subgroup of $G$. The coarse structure $\mathcal{E}_{\mathcal{P}(H)}$ induced by $\mathcal{P}(H)$ coincides with the subspace coarse structure inherited from $G$. 
\end{lemma}

\begin{proof}
Since $H$ is a cocompact subgroup of $G$, we can choose a compact subset $K \subseteq G$ containing $1_G$ such that $G = HK$.

Assume $A$ is bounded with respect to $\mathcal{E}_{\mathcal{P}(H)}$. Let $U$ be an arbitrary open neighbourhood of $1_G$ in $G$. The intersection $V = U \cap H$ constitutes an open neighbourhood of $1_H$ in the subspace topology. Because $A$ is bounded in $\mathcal{E}_{\mathcal{P}(H)}$ and $H$ is Polish, by Lemma \ref{lem:algebraic_boundedness}, there exists a finite set $F_H \subseteq H$ and an integer $m \ge 1$ such that $A \subseteq (F_H V)^m$. Since $V \subseteq U$ and $F_H \subseteq G$, we directly obtain $A \subseteq (F_H U)^m$. Therefore, $A$ is bounded with respect to the subspace coarse structure inherited from $G$.

Conversely, assume $A$ is bounded with respect to the subspace coarse structure. To prove that $A$ is bounded in $\mathcal{E}_{\mathcal{P}(H)}$, it suffices to show that for any open symmetric neighbourhood $V$ of $1_H$, there exist finite subsets $F, L \subseteq H$ and an integer $m \ge 1$ such that $A \subseteq F(LV)^m$. 
Indeed, for any continuous left-invariant pseudometric $d \in \mathcal{P}(H)$, the open unit ball $V_d = \{h \in H \mid d(1_H, h) < 1\}$ forms an open symmetric neighbourhood of $1_H$. By assumption, there exist finite sets $F, L \subseteq H$ and an integer $m \ge 1$ such that $A \subseteq F(LV_d)^m$. Because finite sets are $d$-bounded and the product of bounded sets remains bounded under a left-invariant pseudometric, the set $F(LV_d)^m$ is bounded with respect to $d$. Consequently, $A$ is bounded with respect to $d$. Since $d \in \mathcal{P}(H)$ was arbitrary, $A$ is bounded in $\mathcal{E}_{\mathcal{P}(H)}$.

Let $V$ be an arbitrary open symmetric neighbourhood of $1_H$. Since $K$ is compact, $K K^{-1} \cap H$ is a compact subset of $H$, meaning there exists a finite set $L \subseteq H$ such that $K K^{-1} \cap H \subseteq L V$.

We claim there exists an open symmetric neighbourhood $W$ of $1_G$ such that $(K W K^{-1}) \cap H \subseteq L V$. 
Because $H$ is a closed subspace of $G$ and $L V$ is open in $H$, the union $O = L V \cup (G \setminus H)$ constitutes an open set in $G$. By our previous deduction, $K K^{-1} \subseteq O$. Consider the continuous multiplication map $\Phi \colon (K \times K) \times G \to G$ defined by $\Phi((k_1, k_2), g) = k_1 g k_2^{-1}$. Since $\Phi((K \times K) \times \{1_G\}) = K K^{-1} \subseteq O$ and the topological product $K \times K$ is compact, the Tube Lemma guarantees the existence of an open symmetric neighbourhood $W$ of $1_G$ such that $\Phi((K \times K) \times W) = K W K^{-1} \subseteq O$. Intersecting both sides with $H$ yields $(K W K^{-1}) \cap H \subseteq O \cap H = L V$, proving the claim.

Because $A$ is bounded in the subspace coarse structure, there exists a finite set $F \subseteq G$ and an integer $m \ge 1$ such that $A \subseteq F W^m$. Since $G=HK$, we can choose a finite subset $F_H \subseteq H$ such that $F \subseteq F_H K$. Thus, $A \subseteq F_H K W^m$.

To project this bounded set back into $H$, we define the fundamental domain projection map $P_K \colon 2^G \to 2^H$ by:
\[ P_K(B) = B K^{-1} \cap H \]
This projection satisfies the following structural properties:
\begin{enumerate}
    \item[(P1)] For any subset $B \subseteq G$, $B \subseteq P_K(B) K$.
    \item[(P2)] For any subset $S \subseteq H$, $P_K(S K W) = S (K W K^{-1} \cap H)$.
    \item[(P3)] If $B_1, B_2 \subseteq G$ such that $B_1 \subseteq B_2$, then $P_K(B_1) \subseteq P_K(B_2)$.
    \item[(P4)] If $B \subseteq H$, then $B \subseteq P_K(B)$ (since $1_G \in K$).
\end{enumerate}

Since $A \subseteq F_H K W^m$, applying the monotonicity property (P3) yields $P_K(A) \subseteq P_K(F_H K W^m)$. Because $A \subseteq H$, property (P4) ensures $A \subseteq P_K(A)$, combining to give $A \subseteq P_K(F_H K W^m)$. We claim that $P_K(F_H K W^m) \subseteq F_H (L V)^m$.

We proceed inductively. For the base case $m=1$, since $K W K^{-1} \cap H \subseteq L V$, multiplying by $F_H$ from the left and applying property (P2) yields:
\[ P_K(F_H K W) = F_H (K W K^{-1} \cap H) \subseteq F_H (L V). \]
Now, assume the induction hypothesis holds for $m-1$:
\[ P_K(F_H K W^{m-1}) \subseteq F_H (L V)^{m-1} \]
Multiplying this inclusion from the right by $K$ and applying property (P1) to the left-hand side yields:
\[ F_H K W^{m-1} \subseteq P_K(F_H K W^{m-1}) K \subseteq F_H (L V)^{m-1} K \]
Multiplying by $W$ from the right and applying the monotonic projection property (P3) implies:
\[ P_K(F_H K W^m) \subseteq P_K(F_H (L V)^{m-1} K W) \]
Since $F_H (L V)^{m-1} \subseteq H$, evaluating the right-hand side via property (P2) yields:
\[ P_K(F_H (L V)^{m-1} K W) = F_H (L V)^{m-1} (K W K^{-1} \cap H) \]
Finally, applying our neighbourhood containment $K W K^{-1} \cap H \subseteq L V$ results in:
\[ F_H (L V)^{m-1} (K W K^{-1} \cap H) \subseteq F_H (L V)^{m-1} (L V) = F_H (L V)^m \]
Combining these steps completes the induction, yielding:
\[ A \subseteq P_K(F_H K W^m) \subseteq F_H (L V)^m \]
as desired.
\end{proof}

The coincidence of these bornologies permits the construction of a valid coarse inverse, yielding a resolution to Problem B.7.

\begin{theorem}[Resolution of Problem B.7] \label{thm:cocompact_subgroup}
Let $G$ be a Polish group. Suppose $H$ is a closed subgroup of $G$ that is cocompact. Then $H$, equipped with the coarse structure $\mathcal{E}_{\mathcal{P}(H)}$ induced by its pseudometrics, is coarsely equivalent to $G$. In particular, the inclusion map $\iota \colon H \hookrightarrow G$ is a coarse equivalence.
\end{theorem}

\begin{proof}
Because $H$ is cocompact in $G$, there exists a compact subset $K \subseteq G$ such that $G = HK$. In the left-invariant coarse structure of any topological group, every compact set is bounded, meaning $K \in \mathcal{B}$. 

By Proposition \ref{prop:bounded_index}, the factorisation $G = HK$ with bounded $K$ guarantees that $H$, when equipped with the subspace coarse structure inherited from $G$, is coarsely equivalent to $G$ via the inclusion map $\iota \colon H \hookrightarrow G$.

Furthermore, Lemma \ref{lem:intrinsic_extrinsic} establishes that because $H$ is a closed and cocompact subgroup, its coarse structure $\mathcal{E}_{\mathcal{P}(H)}$ coincides exactly with this subspace coarse structure. Therefore, $H$ equipped with its intrinsic bornology is coarsely equivalent to $G$, resolving the problem.
\end{proof}
\section{Example: The Urysohn Universal Space}

We illustrate the necessity of the coarse generalisation by analysing the isometry group of the Urysohn universal metric space, denoted $G = \operatorname{Isom}(\mathbb{U})$. Equipped with the topology of pointwise convergence, $G$ is a non-discrete, non-locally compact Polish group where the classical framework of proper discontinuity does not apply.

Following Rosendal \cite{Rosendal}, the continuous left-invariant bornology on $\operatorname{Isom}(\mathbb{U})$ is generated by the displacement of bounded sets. A subset $A \subseteq G$ is coarsely bounded if and only if for every metrically bounded set $B \subseteq \mathbb{U}$, the uniform displacement $\sup_{g \in A} \sup_{x \in B} d(x, gx)$ is finite. As an application of our framework, we recover the large-scale geometry of this action, providing an alternative to the treatment in \cite[Proposition 3.8]{Rosendal}.

\begin{proposition}[cf. Rosendal {\cite[Proposition 3.8]{Rosendal}}]
The canonical action of $G = \operatorname{Isom}(\mathbb{U})$ on $\mathbb{U}$ is coarsely proper and coarsely cobounded. 
\end{proposition}

\begin{proof}
Let $B \subseteq \mathbb{U}$ be a metrically bounded set. The transporter is $(B \mid B)_G = \{g \in G \mid gB \cap B \neq \emptyset\}$. For any $g \in (B \mid B)_G$, there exists a point $z \in B$ such that $gz \in B$. 

For any $x \in B$, we apply the triangle inequality:
\[ d(x, gx) \le d(x, z) + d(z, gz) + d(gz, gx). \]
Since $x, z \in B$, $d(x, z) \le \operatorname{diam}(B)$. Since $z, gz \in B$, $d(z, gz) \le \operatorname{diam}(B)$. Because $g$ acts as an isometry, $d(gz, gx) = d(z, x) \le \operatorname{diam}(B)$. 

Summing these bounds yields $d(x, gx) \le 3\operatorname{diam}(B)$ for all $x \in B$ and all $g \in (B \mid B)_G$. Therefore:
\[ \sup_{g \in (B \mid B)_G} \sup_{x \in B} d(x, gx) \le 3\operatorname{diam}(B) < \infty. \]
This proves that the transporter $(B \mid B)_G$ is a coarsely bounded subset of $\operatorname{Isom}(\mathbb{U})$ under its left-invariant bornology. Thus, the action is coarsely proper. Furthermore, because $\mathbb{U}$ is ultrahomogeneous, $\operatorname{Isom}(\mathbb{U})$ acts transitively, meaning the action is coarsely cobounded.
\end{proof}

\end{document}